\documentclass[reqno,11pt]{amsart}
\DeclareMathOperator*{\esssup}{\mathrm{ess\,sup}}

\usepackage{amssymb,amsthm}
\usepackage{amsmath}
\usepackage{amsfonts}
\usepackage{dsfont}

\usepackage[a4paper,  margin=3.2cm]{geometry}

\usepackage[active]{srcltx}
\makeatletter\@addtoreset{equation}{section}\makeatother

\newtheorem{theorem}{Theorem}[section]

\newtheorem{lemma}[theorem]{Lemma}

\newtheorem{proposition}[theorem]{Proposition}
\newtheorem{assumption}[theorem]{Assumption}
\newtheorem{definition}[theorem]{Definition}
\newtheorem{remark}[theorem]{Remark}
\numberwithin{equation}{section}

\title{Evolution of states in a continuum migration model}

\author{Yuri  Kondratiev}
\address{Fakut\"at f\"ur Mathematik, Universit\"at Bielefeld, Bielefeld D-33615, Germany and Interdisciplinary Center
for Complex Systems, Dragomanov University, Kyiv, Ukraine}
\email{kondrat@math.uni-bielefeld.de}

\author{ Yuri  Kozitsky}

\address{Instytut Matematyki, Uniwersytet Marii Curie-Sk{\l}odowskiej, 20-031 Lublin, Poland}
\email{jkozi@hektor.umcs.lublin.pl}

\keywords{Markov evolution, competition kernel, Poisson random
field}
\begin{document}

\subjclass{60J80; 92D25; 82C22}%

\begin{abstract}

The Markov evolution of states of a continuum migration model is
studied. The model describes an infinite system of entities placed
in $\mathds{R}^d$ in which the constituents appear (immigrate) with
rate $b(x)$ and disappear, also due to competition. For this model,
we prove the existence of the evolution of states $\mu_0 \mapsto
\mu_t$ such that the  moments $\mu_t(N_\Lambda^n)$, $n\in
\mathds{N}$, of the number of entities in compact $\Lambda \subset
\mathds{R}^d$ remain bounded for all $t>0$. Under an additional
condition, we prove that the density of entities and the second
correlation function remain bounded globally in time.

\end{abstract}

\maketitle

\section{Introduction}
 We study the Markov dynamics of an infinite
system of point entities placed in $\mathds{R}^d$, $d\geq 1$, which
appear (immigrate) with space-dependent rate $b(x)\geq 0$, and
disappear. The rate of disappearance of the entity located at a
given $x\in \mathds{R}^d$ is the sum of the intrinsic disappearance
rate $m(x)\geq 0$ and the part related to the interaction with the
existing community, which is interpreted as \emph{competition}
between the entities. The phase space is the set $\Gamma$ of all
subsets $\gamma \subset \mathds{R}^d$ such that the set
$\gamma_\Lambda:=\gamma\cap\Lambda$ is finite whenever $\Lambda
\subset \mathds{R}^d$ is compact. For each such $\Lambda$, one
defines the counting map $\Gamma \ni \gamma \mapsto |\gamma_\Lambda|
:= \#\{\gamma\cap\Lambda\}$, where the latter denotes cardinality.
Thereby, one introduces the subsets $\Gamma^{\Lambda,n}:=\{ \gamma
\in \Gamma : |\gamma_\Lambda| = n\}$, $n \in \mathds{N}_0$, and
equips $\Gamma$ with the $\sigma$-field generated by all such
$\Gamma^{\Lambda,n}$. This allows for considering probability
measures on $\Gamma$ as states of the system. Among them there are
Poissonian states in which the entities are independently
distributed over $\mathds{R}^d$,  see \cite[Chapter 2]{Kingman}.
They may serve as reference states for studying correlations between
the positions of the entities. For the nonhomogeneous Poisson
measure $\pi_\varrho$ with density $\varrho:\mathds{R}^d \to
\mathds{R}_{+}:=[0, +\infty)$, $n \in \mathds{N}_0$ and every
compact $\Lambda$, one has
\begin{equation}
  \label{J1}
\pi_\varrho (\Gamma^{\Lambda,n}) = \langle \varrho \rangle_\Lambda^n
\exp\left( - \langle  \varrho \rangle_\Lambda \right)/ n!, \qquad
\langle \varrho \rangle_\Lambda:= \int_{\Lambda} \varrho(x) dx .
\end{equation}
By (\ref{J1}) one readily gets the $\pi_\varrho$-expected value of
the number $N_\Lambda$ of entities contained in $\Lambda$ in the
form
\begin{gather}
  \label{J1a}
\pi_\varrho(N_\Lambda) =  \int_{\Gamma} |\gamma_\Lambda|
\pi_{\varrho} (d\gamma)= \langle \varrho \rangle_\Lambda.
\end{gather}
The case of $\varrho\equiv \varkappa>0$ corresponds to the
homogeneous Poisson measure $\pi_\varkappa$.

The counting map $\Gamma \ni \gamma \mapsto |\gamma|$ can also be
defined for $\Lambda = \mathds{R}^d$. Then the set of \emph{finite}
configurations
\begin{equation}
 \label{J3a}
 \Gamma_0:= \bigcup_{n\in \mathds{N}_0}\{
\gamma \in \Gamma: |\gamma|=n\}
\end{equation}
is measurable. In a state with the property $\mu(\Gamma_0)=1$, the
system  is ($\mu$-almost surely) \emph{finite}. By (\ref{J1}) one
gets that either $\pi_\varrho(\Gamma_0) =1$ or
$\pi_\varrho(\Gamma_0) =0$, depending on whether or not $\varrho$ is
globally integrable. As $\pi_\varkappa(\Gamma_0) =0$,  the system in
state $\pi_\varkappa$ is infinite. The use of infinite
configurations for modeling large finite populations is as a rule
justified, see, e.g., \cite{Cox}, by the argument that in such a way
one gets rid of the boundary and size effects. Note that a finite
system with dispersal -- like the one studied in
\cite{DimaN2,FKKK,KK} -- being placed in a noncompact habitat always
disperse to fill its empty parts, and thus is \emph{developing}.
Infinite configurations are supposed to model \emph{developed}
populations. In this work, we shall consider infinite systems.

To characterize  states on $\Gamma$ one employs {\it observables} --
appropriate functions $F:\Gamma \rightarrow \mathds{R}$. Their
evolution is obtained from the Kolmogorov equation
\begin{equation}
 \label{R2}
\frac{d}{dt} F_t = L F_t , \qquad F_t|_{t=0} = F_0, \qquad t>0,
\end{equation}
where the generator $L$ specifies the model. The states' evolution
is then obtained from the Fokker--Planck equation
\begin{equation}
 \label{R1}
\frac{d}{dt} \mu_t = L^* \mu_t, \qquad \mu_t|_{t=0} = \mu_0,
\end{equation}
related to that in (\ref{R2}) by the duality $\mu_t(F_0) =
\mu_0(F_t)$, where
\[
\mu(F) := \int_{\Gamma} F(\gamma) \mu(d \gamma).
\]
The model that we study in this work is specified by the following
\begin{eqnarray}
 \label{L}
\left(L F \right)(\gamma) & = & \sum_{x\in \gamma}\left( m(x) +
\sum_{y\in \gamma\setminus x} a (x-y)  \right) \left[F(\gamma \setminus x) - F(\gamma) \right]\\[.2cm]
& + & \int_{\mathds{R}^d} b(x) \left[F(\gamma \cup x) - F(\gamma)
\right]dx. \nonumber
\end{eqnarray}
Here $b(x)$ is the {\it immigration} rate, $m(x) \geq 0$ is the
intrinsic emigration (mortality) rate, and $a\geq 0$ is the
\emph{competition kernel}.
\begin{assumption}
  \label{Ass1}
The competition kernel $a$  is continuous and belongs to $L^1
(\mathds{R}^d) \cap L^\infty (\mathds{R}^d)$. The immigration and
mortality rates $b$ and $m$ are continuous and bounded.
\end{assumption}
According to this we set
\begin{gather}
  \label{J6a}
\langle a\rangle = \int_{\mathds{R}^d} a (x ) dx,
\qquad \|a\| = \sup_{x\in \mathds{R}^d} a(x) , \\[.2cm] \|b\| = \sup_{x\in
\mathds{R}^d} b(x), \qquad  \|m\| = \sup_{x\in \mathds{R}^d} m(x) .
\nonumber
\end{gather}
If one takes in (\ref{L}) $a\equiv 0$, the model becomes exactly
soluble, see subsection 2.3 below. This means that the evolution can
be constructed explicitly for each initial state $\mu_0$. Assuming
that $\mu_0 (N_{\Lambda}^n) < \infty$ for all $n\in \mathds{N}$, one
can get the information about the time dependence of such moments.
For $m(x) \geq m_*>0$ for all $x\in \mathds{R}^d$, one obtains that
\begin{equation*}
\forall t>0 \qquad  \mu_t (N_{\Lambda}^n) \leq C_\Lambda^{(n)},
\end{equation*}
holding for each compact $\Lambda$. Otherwise, all the moments
$\mu_t (N_{\Lambda}^n)$ are increasing ad infinitum as $t\to
+\infty$. If the initial state is $\pi_{\varrho_0}$, then $\mu_t =
\pi_{\varrho_t}$ with $\varrho_t (x) = \varrho_0(x) +b(x) t$ for all
$x$ such that $m(x)=0$, cf (\ref{T3}) below.  In \cite{DimaR}, for
the model (\ref{L}) with $m\equiv 0$ and a nonzero $a$ satisfying a
certain (quite burdensome) condition, it was shown that $\mu_t
(N_\Lambda)  \leq C  {\rm V}(\Lambda)$ for an appropriate constant
and large enough values of the Euclidean volume  ${\rm V}(\Lambda)$,
provided the evolution of states $\mu_0 \mapsto \mu_t$ exists.

In this article, assuming that the initial state $\mu_0$ is
sub-Poissonian, see Definition \ref{J1df} below, we prove that the
evolution of states $\mu_0 \mapsto \mu_t$, $t>0$, exists (Theorem
\ref{1tm}) and is such that $\mu_t (N_\Lambda^n) \leq
C^{(n)}_\Lambda$ for each $t>0$ (Theorem \ref{2tm}) and all $m$,
including the case $m\equiv0$. Moreover, if the correlation
functions $k_{\mu_0}^{(n)}$, $n\in \mathds{N}$, of the initial state
are continuous,  see subsection 2.1 below, then all
$k_{\mu_t}^{(n)}$, $n\in \mathds{N}$ are also continuous and such
that
\begin{gather}
  \label{J1d}
k^{(1)}_{\mu_t}(x) \leq C_1,
\qquad
k_{\mu_t}^{(2)} (x, y) d x d y \leq C_2,
\end{gather}
holding (with some positive $C_1$ and $C_2$) for all $t>0$ and all values of the spatial variables.

The structure of the article is as follows. In Section 2, we
introduce the necessary technicalities and then formulate the
results: Theorems \ref{1tm} and \ref{2tm}. Thereafter, we make a
number of comments to them. In Sections 3 and 4,  we present the
proofs of Theorems \ref{1tm} and \ref{2tm}, respectively.

\section{Preliminaries and the Results}

We begin by outlining some technical aspects of this work -- a more
detailed description of them can be found in
\cite{DimaN2,FKKK,KK,Tobi} and in the literature quoted therein.

By $\mathcal{B}(\mathds{R})$ we denote the sets of all Borel subsets
of $\mathds{R}$. The configuration space $\Gamma$ is equipped with
the vague topology, see \cite{Tobi}, and thus with the corresponding
Borel $\sigma$-field $\mathcal{B}(\Gamma)$, which makes it a
standard Borel space. Note that $\mathcal{B}(\Gamma)$ is exactly the
$\sigma$-field generated by the sets $\Gamma^{\Lambda,n}$, mentioned
in Introduction.  By $\mathcal{P}(\Gamma)$ we denote the set of all
probability measures on $(\Gamma, \mathcal{B}(\Gamma))$.

\subsection{Correlation functions}

Like in \cite{FKKK,KK}, the evolution of states will be described by
means of correlation functions  without the direct use of
(\ref{R1}). To explain the essence of this approach let us consider
the set $\varTheta$ of all compactly supported continuous functions
$\theta:\mathbb{R}^d\to (-1,0]$. For a state, $\mu$, its {\it
Bogoliubov} functional, cf. \cite{TobiJoao}, is
\begin{equation}
  \label{I1}
B_\mu (\theta) = \int_{\Gamma} \prod_{x\in \gamma} ( 1 + \theta (x))
\mu( d \gamma), \qquad \theta \in \varTheta.
\end{equation}
For
the homogeneous Poisson measure $\pi_\varkappa$,  it takes the form
\begin{equation*}
B_{\pi_\varkappa} (\theta) = \exp\left(\varkappa
\int_{\mathbb{R}^d}\theta (x) d x \right).
\end{equation*}
\begin{definition}
  \label{J1df}
The set of states $\mathcal{P}_{\rm exp}(\Gamma)$ is defined as that
containing all those states $\mu\in \mathcal{P}(\Gamma)$ for which
$B_\mu$ can be continued, as a function of $\theta$, to an
exponential type entire function on $L^1 (\mathbb{R}^d)$. The
elements of $\mathcal{P}_{\rm exp}(\Gamma)$ are called
sub-Poissonian states.
\end{definition}
It can be shown that a given $\mu$ belongs to $\mathcal{P}_{\rm
exp}(\Gamma)$ if and only if its functional $B_\mu$ can be written
down in the form
\begin{eqnarray}
  \label{I3}
B_\mu(\theta) = 1+ \sum_{n=1}^\infty
\frac{1}{n!}\int_{(\mathbb{R}^d)^n} k_\mu^{(n)} (x_1 , \dots , x_n)
\theta (x_1) \cdots \theta (x_n) d x_1 \cdots d x_n,
\end{eqnarray}
where $k_\mu^{(n)}$ is the $n$-th order correlation function of
$\mu$. It is a symmetric element of $L^\infty ((\mathbb{R}^d)^n)$
for which
\begin{equation}
\label{I4}
  \|k^{(n)}_\mu \|_{L^\infty
((\mathbb{R}^d)^n)} \leq C \exp( \vartheta n), \qquad n\in
\mathbb{N}_0,
\end{equation}
with some $C>0$ and $\vartheta \in \mathbb{R}$. Note that
$k_{\pi_\varkappa}^{(n)} (x_1 , \dots , x_n)= \varkappa^n$. Note
also that (\ref{I3}) resembles the Taylor expansion of the
characteristic function of a probability measure. In view of this,
$k^{(n)}_\mu$ are also called (factorial) \emph{moment functions},
cf e.g., \cite{Mu}.

Recall that $\Gamma_0$ -- the set of all finite $\gamma \in \Gamma$
defined in (\ref{J3a}) -- is an element of $\mathcal{B}(\Gamma)$. A
function $G:\Gamma_0 \to \mathds{R}$ is
$\mathcal{B}(\Gamma)/\mathcal{B}(\mathds{R} )$-measurable, see
\cite{FKKK}, if and only if, for each $n\in \mathds{N}$, there
exists a symmetric Borel function $G^{(n)}: (\mathds{R}^{d})^{n} \to
\mathds{R}$ such that
\begin{equation}
 \label{7}
 G(\eta) = G(\eta) = G^{(n)} ( x_1, \dots , x_{n}), \quad {\rm for} \ \eta = \{ x_1, \dots , x_{n}\}.
\end{equation}
\begin{definition}
  \label{Gdef}
A measurable function $G:\Gamma_0 \to \mathds{R}$  is said to have bounded
support if: (a) there exists $\Lambda \in \mathcal{B}_{\rm b}
(\mathds{R}^d)$ such that $G(\eta) = 0$ whenever $\eta\cap
(\mathds{R}^d \setminus \Lambda)\neq \emptyset$; (b) there exists
$N\in \mathds{N}_0$ such that $G(\eta)=0$ whenever $|\eta|
>N$.  By $\Lambda(G)$ and $N(G)$ we denote the
smallest $\Lambda$ and $N$ with the properties just mentioned. By
$B_{\rm bs}(\Gamma_0)$ we denote the set of all such functions.
\end{definition}
Set $\mathcal{B}(\Gamma_0) = \{A\in \mathcal{B}(\Gamma): A \subset
\Gamma_0\}$. The Lebesgue-Poisson measure $\lambda$ on $(\Gamma_0,
\mathcal{B}(\Gamma_0))$ is defined by the following formula
\begin{eqnarray}
\label{8} \int_{\Gamma_0} G(\eta ) \lambda ( d \eta)  = G(\emptyset)
+ \sum_{n=1}^\infty \frac{1}{n! } \int_{(\mathds{R}^d)^{n}} G^{(n)}
( x_1, \dots , x_{n} ) d x_1 \cdots dx_{n},
\end{eqnarray}
holding for all $G\in B_{\rm bs}(\Gamma_0)$. Like in (\ref{7}), we
introduce $k_\mu : \Gamma_0 \to \mathds{R}$ such that $k_\mu(\eta) =
k^{(n)}_\mu (x_1, \dots , x_n)$ for $\eta = \{x_1, \dots , x_n\}$,
$n\in \mathds{N}$. We also set $k_\mu(\emptyset)=1$. With the help
of the measure introduced in (\ref{8}), the expressions for $B_\mu$ in
(\ref{I1}) and (\ref{I3}) can be combined into the following formulas
\begin{eqnarray}
  \label{1fa}
 B_\mu (\theta)& = & \int_{\Gamma_0} k_\mu(\eta) \prod_{x\in \eta} \theta (x) \lambda (d\eta)=: \int_{\Gamma_0} k_\mu(\eta) e( \eta; \theta) \lambda (d \eta)
 \\[.2cm]
 & = &  \int_{\Gamma} \prod_{x\in \gamma} (1+ \theta (x)) \mu (d \gamma) =: \int_{\Gamma} F_\theta (\gamma) \mu(d
 \gamma). \nonumber
\end{eqnarray}
Thereby, one can transform  the action of $L$ on $F$, see
(\ref{L}), to the action of $L^\Delta$ on $k_\mu$ according to the
rule
\begin{equation}
  \label{1g}
\int_{\Gamma}(L F_\theta) (\gamma) \mu(d \gamma) = \int_{\Gamma_0}
(L^\Delta k_\mu) (\eta) e(\eta;\theta)
 \lambda (d \eta).
\end{equation}
This will allow us to pass from (\ref{R1}) to the corresponding
Cauchy problem for the correlation functions
\begin{equation}
  \label{J7}
\frac{d}{dt} k_t = L^\Delta k_t, \qquad k_t|_{t=0} = k_{\mu_0}.
\end{equation}
By (\ref{1g}) the action of $L^\Delta$ is
\begin{equation}
  \label{J10}
\left( L^\Delta k \right) (\eta) = (L^{\Delta,-}k)(\eta) +
\sum_{x\in \eta} b(x) k(\eta \setminus x),
\end{equation}
where
\begin{equation}
 \label{J8a}
(L^{\Delta,-}k)(\eta)  = - E(\eta) k(\eta) -
\int_{\mathds{R}^d}\left(\sum_{y\in \eta} a (x-y) \right) k(\eta\cup
x)d x,
 \end{equation}
and
\begin{equation}
  \label{J9}
  E(\eta) = \sum_{x\in \eta}m(x) + \sum_{x\in \eta} \sum_{y\in \eta\setminus x} a(x-y).
\end{equation}
 In the next
subsection, we introduce the spaces where we are going to define
(\ref{J7}).

\subsection{The Banach spaces}
By (\ref{I3}) and (\ref{1fa}), it follows that $\mu \in
\mathcal{P}_{\rm exp}(\Gamma)$ implies
\begin{equation*}
 |k_\mu (\eta)| \leq C \exp( \vartheta  |\eta|),
\end{equation*}
holding for $\lambda$-almost all $\eta\in \Gamma_0$, some $C>0$, and
$\vartheta\in \mathds{R}$. In view of this, we set
\begin{equation}
 \label{18}
\mathcal{K}_\vartheta := \{ k:\Gamma_0\to \mathds{R}:
\|k\|_\vartheta <\infty\},
\end{equation}
where
\begin{equation}
  \label{17a}
 \|k\|_\vartheta = \esssup_{\eta \in \Gamma_0}\left\{ |k_\mu (\eta)| \exp\big{(} - \vartheta
  |\eta| \big{)} \right\}.
\end{equation}
Clearly, (\ref{18}) and (\ref{17a}) define a Banach space. In the following, we use the ascending scale of such
spaces $\mathcal{K}_\vartheta$, $\vartheta \in \mathds{R}$, with the
property
\begin{equation}
  \label{19}
\mathcal{K}_\vartheta \hookrightarrow \mathcal{K}_{\vartheta'},
\qquad \vartheta < \vartheta',
\end{equation}
where $\hookrightarrow$  denotes continuous embedding.

For $G\in B_{\rm bs}(\Gamma_0)$, we set
\begin{equation}
  \label{9a}
(KG)(\gamma) = \sum_{\eta \Subset \gamma} G(\eta),
\end{equation}
where $\Subset$ indicates that the summation is taken over all
finite subsets. It satisfies, see Definition \ref{Gdef},
\begin{equation*}
|(KG)(\gamma)| \leq \left( 1 + |\gamma\cap\Lambda (G)|\right)^{N(G)}.
\end{equation*}
The latter means that $\mu(KG) < \infty$ for each $\mu\in
\mathcal{P}_{\rm exp}(\Gamma)$. By (\ref{1fa}) this yields
\begin{equation}
 \label{J7a}
\langle \! \langle G, k_\mu \rangle \! \rangle := \int_{\Gamma_0}
G(\eta) k_\mu(\eta) \lambda (d \eta) =\mu(KG)
  < \infty.
\end{equation}
Set
\begin{equation}
  \label{9g}
B^\star_{\rm bs} (\Gamma_0) =\{ G\in B_{\rm bs}(\Gamma_0):
(KG)(\gamma) \geq 0 \ {\rm for} \ {\rm all} \ \gamma\in \Gamma\}.
\end{equation}
By \cite[Theorems 6.1 and 6.2 and Remark 6.3]{Tobi} one can prove
the next statement.
\begin{proposition}
  \label{Gpn}
Let  a measurable function $k : \Gamma_0 \to \mathds{R}$  have the
following properties:
\begin{eqnarray}
  \label{9h}
& (a) & \ \langle \! \langle G, k \rangle \!\rangle \geq 0, \qquad
{\rm for} \ {\rm all} \ G\in B^\star_{\rm bs} (\Gamma_0);\\[.2cm]
& (b) & \ k(\emptyset) = 1; \qquad (c) \ \ k(\eta) \leq
 C^{|\eta|} ,
\nonumber
\end{eqnarray}
with (c) holding for some $C >0$ and $\lambda$-almost all $\eta\in
\Gamma_0$. Then there exists a unique state $\mu \in
\mathcal{P}_{\rm exp}(\Gamma)$ for which $k$ is the correlation
function.
\end{proposition}
Set, cf (\ref{9g}),
\begin{equation}
  \label{19a}
\mathcal{K}^\star_\vartheta =\{k\in \mathcal{K}_\vartheta: \langle
\! \langle G,k \rangle \! \rangle \geq 0 \ {\rm for} \ {\rm all} \
G\in B^\star_{\rm bs} (\Gamma_0)\},
\end{equation}
which is a subset of the cone
\begin{equation}
  \label{19b}
\mathcal{K}^+_\vartheta =\{k\in \mathcal{K}_\vartheta: k(\eta) \geq
0 \ \ {\rm for} \  \lambda-{\rm almost} \ {\rm all} \ \eta \in
\Gamma_0\}.
\end{equation}
By Proposition \ref{Gpn} it follows that each $k\in
\mathcal{K}^\star_\vartheta$ such that $k(\emptyset) = 1$ is the
correlation function of a unique state $\mu\in \mathcal{P}_{\rm
exp}(\Gamma)$. Then we define
\begin{equation*}
\mathcal{K} = \bigcup_{\vartheta \in \mathds{R}}
\mathcal{K}_\vartheta, \qquad \mathcal{K}^\star = \bigcup_{\vartheta
\in \mathds{R}} \mathcal{K}_\vartheta^\star.
\end{equation*}
As a sum of Banach spaces, the linear space $\mathcal{K}$ is
equipped with the corresponding inductive topology that turns it
into a locally convex space.

\subsection{Without competition}

The version of  (\ref{L}) with  $a\equiv 0$ is known as the
Surgailis model, see \cite{Sur} and the discussion in \cite{DimaR}.
This model is exactly soluble, which means that the solution of
(\ref{J7}) can be written down explicitly in the following form
\begin{equation}
  \label{T5}
k_t (\eta) = \sum_{\xi\subset \eta}e(\xi; \phi_t) e(\eta \setminus
\xi; \psi_t) k_{\mu_0}(\eta\setminus \xi),
\end{equation}
where
\begin{eqnarray}
  \label{T3}
\psi_t (x)& = & e^{- m(x) t}, \qquad \quad e(\xi; \phi)= \prod_{x\in
\xi}\phi(x), \\[.2cm] \nonumber \phi_t(x) & = & \left\{ \begin{array}{ll} \left( 1 - e^{- m(x) t }
\right)\frac{b(x)}{m(x)} \qquad  &{\rm for} \ \ m(x)>0,\\[.3cm]b(x)t \qquad &{\rm for}  \ \ m(x)=0. \end{array} \right.
\end{eqnarray}
The corresponding state $\mu_t$ has the Bogoliubov functional
\begin{equation}
 \label{BF}
B_{\mu_t} (\theta) = \exp\left( \int_{\mathds{R}^d} \theta (x)
\phi_t (x) d x \right)B_{\mu_0} (\theta \psi_t),
\end{equation}
which one obtains from (\ref{1fa}) and (\ref{T5}). This formula can
be used to extend the evolution $\mu_0\mapsto \mu_t$ to all
$\mu_0\in \mathcal{P}(\Gamma)$. Indeed, for each $t>0$ and
$\theta\in \varTheta$, cf (\ref{I1}), we have that $\theta \psi_t
\in \varTheta$, and hence $B_{\mu_0} (\theta \psi_t)$ is the
Bogoliubov functional of a certain state.\footnote{This state is an
independent thinning of $\mu_0$.} The same is true for the left-hand
side of (\ref{BF}), and the state $\mu_t$  can be considered as a
weak solution of the corresponding Fokker-Planck equation
(\ref{R1}).

If the initial
state is Poissonian with density $\varrho_0 (x)$, by (\ref{BF}) the state $\mu_t$ is also Poissonian with the density
$$\varrho_t (x) = \psi_t (x) \varrho_0 (x) + \phi_t(x).$$
If $m(x) \geq m_*>0$ for some $m_*$ and all $x\in \mathds{R}^d$,
then the solution in (\ref{T5}) lies in $\mathcal{K}_{\vartheta_*}$
for all $t>0$. Here
\begin{equation}
  \label{T1}
  \vartheta_* = \max\{ \vartheta_0; \log (\|b\|/m_*)\}.
\end{equation}
Otherwise, the solution in (\ref{T5}) is unboundedly increasing in
$t$. If, for some compact $\Lambda$, $m(x)=0$ for $x\in \Lambda$,
then by (\ref{T5}) and (\ref{T3}) we get
\begin{equation*}
 k_t^{(1)} (x) = k_{\mu_0}^{(1)} (x) + b(x) t , \qquad x \in
 \Lambda,
\end{equation*}
that by (\ref{J1a}), (\ref{9a}) and (\ref{J7a}) yields
\begin{gather}
\label{T6}
  \mu_t (N_\Lambda) = \int_{\Gamma} |\gamma_\Lambda| \mu_t ( d \gamma) = \int_{\Gamma}\left( \sum_{x\in \gamma} I_\Lambda (x)\right) \mu_t(d\gamma)
  \\[.2cm] \nonumber =
\int_{\Gamma} (K I_\Lambda)(\gamma) \mu_t(d \gamma) =
\int_{\Lambda}k^{(1)}_t(x) d x =
   \mu_0 (N_\Lambda) + t \int_{\Lambda} b(x) dx,
\end{gather}
where $I_\Lambda$ is the indicator of $\Lambda$. Then $\mu_t
(N_\Lambda)\to +\infty$ as $t\to +\infty$ if $b$ is not identically
zero on $\Lambda$.

\subsection{The statements}

For each $\vartheta \in \mathds{R}$ and $\vartheta'>\vartheta$, the
expressions in (\ref{J10}) and (\ref{J8a}) can be used to define the
corresponding bounded linear operators $L^\Delta_{\vartheta'
\vartheta}$ acting from $\mathcal{K}_\vartheta$ to
$\mathcal{K}_{\vartheta'}$. Their operator norms can be estimated
similarly as in \cite[eqs. (3.11), (3.13)]{KK}, which yields, cf.
(\ref{J6a}),
\begin{equation}
 \label{J12}   \|L^\Delta_{\vartheta'
\vartheta}\| \leq  \frac{4 \| a \| }{e^2 (\vartheta' - \vartheta)^2}
+ \frac{ \|b\| e^{-\vartheta} + \|m\| + \langle a \rangle
e^{\vartheta'}}{e (\vartheta' - \vartheta)}.
\end{equation}
By means of the collection $\{L^\Delta_{\vartheta' \vartheta}\}$
with all $\vartheta\in \mathds{R}$ and $\vartheta'>\vartheta$ we
introduce a continuous linear operator acting on $\mathcal{K}$,
denoted also as $L^\Delta$, and thus define the corresponding Cauchy
problem (\ref{J7}) in this space. By its (global in time) solution
we will mean a continuously differentiable function $[0,+\infty) \ni
t \mapsto k_t \in \mathcal{K}$ such that both equalities in
(\ref{J7}) hold. Our results are given in the following statements,
both based on Assumption \ref{Ass1}.
\begin{theorem}[Existence of evolution]
  \label{1tm}
For each $\mu_0 \in \mathcal{P}_{\rm exp} (\Gamma)$, the problem in
(\ref{J7}) with $L^\Delta:\mathcal{K}\to \mathcal{K}$ as in
(\ref{J10}), (\ref{J8a}) and (\ref{J12}) has a unique solution which
lies in  $\mathcal{K}^\star$ and is such that $k_t(\emptyset)=1$ for
all $t>0$. Therefore, for each $t>0$, there exists a unique state
$\mu_t\in \mathcal{P}_{\rm exp}(\Gamma)$ such that $k_t =
k_{\mu_t}$. Moreover, for all $t>0$, the following holds
\begin{equation}
  \label{T2}
0\leq k_t (\eta) \leq \sum_{\xi\subset \eta}e(\xi; \phi_t) e(\eta
\setminus \xi; \psi_t) k_{\mu_0}(\eta\setminus \xi),
\end{equation}
where $\phi_t$ and $\psi_t$ are as in (\ref{T3}). If the intrinsic
mortality rate satisfies $m(x) \geq m_*
>0$ for all $x\in \mathds{R}^d$, then for all $t>0$ the solution $k_t$  lies in $\mathcal{K}_{\vartheta_*}$
with $\vartheta_*$ is given in (\ref{T1}).
\end{theorem}
\begin{theorem}[Global boundedness]
  \label{2tm}
The states $\mu_t$ mentioned in Theorem \ref{1tm} have the property:
for every $n\in \mathds{N}$ and compact $\Lambda
\subset\mathds{R}^d$, the following holds
\begin{equation}
 \label{T200}
\forall t>0 \qquad  \mu_t (N_\Lambda^n) \leq C_\Lambda^{(n)},
 \end{equation}
with some $C_\Lambda^{(n)}>0$. If $\mu_0$ is such that each $k_{\mu_0}^{(n)}$ is a continuous function,
then so is $k_{\mu_t}^{(n)}$ for all $n\in \mathds{N}$ and $t>0$. Moreover, $k^{(1)}_{\mu_t}$ and
$k^{(2)}_{\mu_t}$
have the properties as in (\ref{J1d}).
\end{theorem}

\subsection{Comments and comparison}

By (\ref{T6}) it follows that the global in time boundedness in the
Surgailis model is possible only if $m(x) \geq m_*>0$ for all $x\in
\mathds{R}^d$. As follows from our Theorem \ref{2tm}, adding
competition to the Surgailis model with the zero intrinsic mortality
rate yields the global in time boundedness. In this case, the
competition rate $a(0)$ appears to be an effective mortality, see
the proof of Theorem \ref{2tm} and (\ref{J39}) in particular. Note
also that the global boundedness as in Theorem \ref{2tm} does not
mean that the evolution $k_{\mu_0} \mapsto k_t$ holds in one and the
same $\mathcal{K}_\vartheta$ with sufficiently large $\vartheta$. It
does if $m(x) \geq m_*>0$. Since Theorem \ref{1tm} covers also the
case $a \equiv 0$, the solution in (\ref{T5}) is unique in the same
sense. A partial result on the global boundedness in the model
discussed here was obtained in \cite[Theorem 1]{DimaR}. Therein,
under quite a strong condition imposed on the competition kernel $a$
(which, in particular, implies that it has infinite range), and
under the assumption that the evolution of states $\mu_0 \mapsto
\mu_t$ exists, there was proved the fact which in the present
notations can be formulated as $\mu_t(N_\Lambda) \leq C_\Lambda$.

\section{The Existence of the Evolution of States}

We follow the line of arguments used in proving Theorem 3.3 in
\cite{KK} and perform the following three steps:
\begin{itemize}
  \item[(i)] Defining the Cauchy  problem  (\ref{J7}) with $k_{\mu_0 }\in \mathcal{K}_{\vartheta_0}$ in a given
  Banach space $\mathcal{K}_\vartheta$ with $\vartheta >\vartheta_0$, see (\ref{18}) and (\ref{19}), and then
  showing that this problem has a unique solution $k_t\in
  \mathcal{K} _\vartheta$ on a bounded time interval
  $[0,T(\vartheta, \vartheta_0))$ (subsection 3.1).
\item[(ii)] Proving that the mentioned solution $k_t$ has
properties (a) and (b) in (\ref{9h}) ((c) follows by the fact that
$k_t\in \mathcal{K}_\vartheta$). Then $k_t \in
\mathcal{K}_\vartheta^\star$ and hence also in
$\mathcal{K}_\vartheta^{+}$, see (\ref{19b}) and  (\ref{19a}). By
Proposition \ref{Gpn} it follows that $k_t$ is the correlation
function of a unique state $\mu_t$ (subsection 3.2).
\item[(iii)]
Constructing a continuation of $k_t$ from $[0,T(\vartheta,
\vartheta_0))$ to all $t>0$ by means of the fact that $k_t \in
\mathcal{K}_\vartheta^{+}$
 (subsection 3.3).
\end{itemize}

\subsection{Solving the Cauchy problem}

We begin by rewriting $L^\Delta$ (given in (\ref{J10}), (\ref{J8a}))
in the following form
\begin{eqnarray}
  \label{A1}
L^\Delta & = & A + B, \\[.2cm]
(Ak)(\eta)& = & - E (\eta) k(\eta) , \nonumber \\[.2cm]
(Bk)(\eta)& = & - \int_{\mathds{R}^d} \left( \sum_{y\in \eta} a(x-y)
\right) k(\eta \cup x) dx + \sum_{x\in \eta} b(x) k(\eta \setminus
x). \nonumber
\end{eqnarray}
For $\vartheta\in \mathds{R}$ and $\vartheta'>\vartheta$, let
$\mathcal{L}(\mathcal{K}_{\vartheta}, \mathcal{K}_{\vartheta'})$ be
the Banach space of all bounded linear operators acting from
$\mathcal{K}_{\vartheta}$ to $\mathcal{K}_{\vartheta'}$. Like in
(\ref{J12}) we define $A_{\vartheta'\vartheta},
B_{\vartheta'\vartheta}\in \mathcal{L}(\mathcal{K}_{\vartheta},
\mathcal{K}_{\vartheta'})$, satisfying
\begin{equation}
  \label{A2}
 \|A_{\vartheta'\vartheta} \|\leq \frac{4\|a\|}{e^2(\vartheta' -
 \vartheta)^2} + \frac{ \|m\|}{e(\vartheta' -
 \vartheta)}, \qquad  \|B_{\vartheta'\vartheta} \|\leq \frac{\|b\|e^{-\vartheta}+  \langle a \rangle e^{\vartheta'}}{e(\vartheta' -
 \vartheta)}.
\end{equation}
Now we set,  see (\ref{J9}),
\begin{equation}
  \label{A2a}
 (S (t) k)( \eta) = \exp\left( - t E(\eta) \right) k(\eta),
 \qquad t\geq 0,
\end{equation}
and then introduce the corresponding $S_{\vartheta'\vartheta}(t) \in
\mathcal{L}(\mathcal{K}_{\vartheta}, \mathcal{K}_{\vartheta'})$,
$t\geq 0$. By the first estimate in (\ref{A2}) one shows that the
map
\begin{equation}
  \label{A3}
[0,+\infty ) \ni t \mapsto S_{\vartheta'\vartheta}(t) \in
\mathcal{L}(\mathcal{K}_{\vartheta}, \mathcal{K}_{\vartheta'})
\end{equation}
is continuous and such that
\begin{equation}
  \label{A4}
  \frac{d}{dt} S_{\vartheta'\vartheta}(t) = A_{\vartheta'
  \vartheta''} S_{\vartheta''\vartheta}(t), \qquad t >0,
\end{equation}
holding for each $\vartheta'' \in (\vartheta, \vartheta')$. Note
that (\ref{A2a}) may be used to define a bounded multiplication
operator, $S_{\vartheta}(t)\in \mathcal{L}(\mathcal{K}_\vartheta) :=
\mathcal{L}(\mathcal{K}_\vartheta,\mathcal{K}_\vartheta)$. However,
in this case the map $[0,+\infty ) \ni t \mapsto S_{\vartheta}(t)
\in \mathcal{L}(\mathcal{K}_{\vartheta})$ would not be continuous.

For $\vartheta$ and $\vartheta'>\vartheta$ as above, we fix some
$\delta < \vartheta' - \vartheta$. Then, for a given $l\in
\mathds{N}$, we divide the interval $[\vartheta, \vartheta']$ into
subintervals with endpoints $\vartheta^s$, $s=0, 1 , \dots , 2l+1$,
as follows. Set $\vartheta^0= \vartheta$, $\vartheta^{2l+1} =
\vartheta'$, and
\begin{eqnarray}
  \label{A5}
  \vartheta^{2s} & = & \vartheta + \frac{s}{l+1}\delta + s \epsilon,
  \qquad \epsilon:= (\vartheta' - \vartheta - \delta)/l,\\[.2cm]
 \vartheta^{2s+1} & = & \vartheta + \frac{s+1}{l+1}\delta + s \epsilon,
  \qquad s =0, 1, \dots , l.\nonumber
\end{eqnarray}
Then, for $t>0$ and
\[
(t, t_1 , \dots , t_l) \in \mathcal{T}_l :=\{(t, t_1 , \dots , t_l)
: 0\leq t_l \leq t_{l-1} \cdots \leq t_1 \leq t\}\subset
\mathds{R}^{l+1},
\]
define
\begin{eqnarray}
  \label{A6}
\Pi^{(l)}_{\vartheta'\vartheta} (t, t_1 , \dots , t_l)& = &
S_{\vartheta'\vartheta^{2l}}(t-t_1) B_{\vartheta^{2l}
\vartheta^{2l-1}} \cdots
S_{\vartheta^{2s+1}\vartheta^{2s}}(t_{l-s}-t_{l-s+1})
 \\[.2cm]
\nonumber &\times& B_{\vartheta^{2s} \vartheta^{2s-1}} \cdots
S_{\vartheta^{3}\vartheta^{2}}(t_{l-1}-t_{l}) B_{\vartheta^{2}
\vartheta^{1}}S_{\vartheta^{i}\vartheta}(t_l).
\end{eqnarray}
By (\ref{A4}), (\ref{A3}) and (\ref{A1}) one can prove the next
statement, cf \cite[Proposition 4.6]{KK}.
\begin{proposition}
  \label{A1pn}
 For each $l\in \mathds{N}$, the operators defined in (\ref{A6})
 have  the properties:
 \begin{itemize}
   \item[(i)] for each $(t, t_1 , \dots , t_l) \in \mathcal{T}_l$, $\Pi^{(l)}_{\vartheta'\vartheta} (t, t_1 , \dots ,
   t_l)$ is in $\mathcal{L}(\mathcal{K}_{\vartheta},
   \mathcal{K}_{\vartheta'})$ and the map
  \[
(t, t_1 , \dots , t_l) \mapsto \Pi^{(l)}_{\vartheta'\vartheta} (t,
t_1 , \dots ,
   t_l)\in \mathcal{L}(\mathcal{K}_{\vartheta},
   \mathcal{K}_{\vartheta'})
  \]
  is continuous;
\item[(ii)] for fixed $t_1 , \dots , t_l$ and each $\varepsilon >0$,
the map
\[
(t_1 , t_1 + \varepsilon) \ni t \mapsto
\Pi^{(l)}_{\vartheta'\vartheta} (t, t_1 , \dots ,
   t_l)\in \mathcal{L}(\mathcal{K}_{\vartheta},
   \mathcal{K}_{\vartheta'})
\]
is continuously differentiable and such that, for each
$\vartheta''\in (\vartheta , \vartheta')$, the following holds
\begin{equation}
  \label{A7}
 \frac{d}{dt} \Pi^{(l)}_{\vartheta'\vartheta} (t, t_1 , \dots ,
   t_l) = A_{\vartheta'
  \vartheta''} \Pi^{(l)}_{\vartheta''\vartheta} (t, t_1 , \dots ,
   t_l).
\end{equation}
 \end{itemize}
\end{proposition}
Define
\begin{equation}
  \label{A8}
T(\vartheta', \vartheta) = \frac{\vartheta' -
\vartheta}{\|b\|e^{-\vartheta} + \langle a \rangle e^{\vartheta'}}.
\end{equation}
Then assume that $k_{\mu_0 }\in \mathcal{K}_{\vartheta_0}$, fix some
$\vartheta_1 >\vartheta_0$, and set
\begin{equation}
  \label{A9}
 \varUpsilon = \{(\vartheta, \vartheta', t): \vartheta_0 \leq \vartheta < \vartheta' \leq \vartheta_1, \quad t < T(\vartheta',
 \vartheta)\}.
\end{equation}
\begin{proposition}
  \label{A2pn}
There exists a family of linear operators, $\{Q_{\vartheta'
\vartheta}(t):  (\vartheta, \vartheta', t) \in \varUpsilon\}$, each
element of which is in the corresponding
$\mathcal{L}(\mathcal{K}_{\vartheta},
   \mathcal{K}_{\vartheta'})$ and has the following properties:
\begin{itemize}
  \item[(i)] the map $[0,T(\vartheta', \vartheta)) \ni t \mapsto Q_{\vartheta'
\vartheta}(t) \in \mathcal{L}(\mathcal{K}_{\vartheta},
   \mathcal{K}_{\vartheta'})$ is continuous and $ Q_{\vartheta'
\vartheta}(0)$ is the embedding $\mathcal{K}_\vartheta
\hookrightarrow \mathcal{K}_{\vartheta'}$;
\item[(ii)] for each $\vartheta'' \in  (\vartheta, \vartheta')$ and
$t < T(\vartheta'', \vartheta)$, the following holds
\begin{equation}
  \label{A10}
  \frac{d}{dt} Q_{\vartheta' \vartheta}(t) =
  L^\Delta_{\vartheta'\vartheta''} Q_{\vartheta'' \vartheta}(t).
\end{equation}
\end{itemize}
\end{proposition}
\begin{proof}
We  go along the line of arguments used in the proof of Lemma 4.5 in
\cite{KK}. Take any $T< T(\vartheta', \vartheta)$ and then pick
$\vartheta'' \in (\vartheta, \vartheta']$ and a positive $\delta<
\vartheta'' - \vartheta$ such that also $T< T_\delta :=
T(\vartheta''-\delta, \vartheta)$. For these values of the
parameters, take $\Pi^{(l)}_{\vartheta''\vartheta}$ as in (\ref{A6})
and then, for $n\in \mathds{N}$,  set
\begin{eqnarray}
 \label{A10a}
 Q_{\vartheta'' \vartheta}^{(n)} (t) = S_{\vartheta''\vartheta} (t)  +
 \sum_{l=1}^n \int_0^t \int_0^{t_1} \cdots \int_0^{t_{l-1}} \Pi^{(l)}_{\vartheta'' \vartheta} (t, t_1 \dots, t_l ) d t_l \cdots d t_1.
\end{eqnarray}
By (\ref{A2a}) and the second estimate in (\ref{A2}) we have from
(\ref{A6}) that
\[
 \|\Pi^{(l)}_{\vartheta'' \vartheta} (t,t_1 , \dots , t_l)\| \leq \left( \frac{l}{e T_\delta}\right)^l,
\]
holding for all $l=1, \dots , n$. By (\ref{A10a}), for $t\in [0,T)$,
this yields
\begin{gather*}
\|Q_{\vartheta'' \vartheta}^{(n)} (t) - Q_{\vartheta''
\vartheta}^{(n-1)} (t) \| \leq
\int_0^t \int_0^{t_1} \cdots \int_0^{t_{n-1}}\| \Pi^{(n)}_{\vartheta'' \vartheta} (t, t_1 \dots, t_n )\| d t_l \cdots d t_n \\[.2cm]
\leq \frac{1}{n!} \left(\frac{n}{e} \right)^n
\left(\frac{T}{T_\delta}\right)^n.
\end{gather*}
Hence, for all $t\in [0,T]$, $\{ Q_{\vartheta'' \vartheta}^{(n)} (t)
\}_{n\in \mathds{N}}$ is a Cauchy sequence in
$\mathcal{L}(\mathcal{K}_{\vartheta},\mathcal{K}_{\vartheta''} )$.
The operator $Q_{\vartheta'' \vartheta} (t)$ in question is then its
limit. The continuity in (i) follows by the fact that the
convergence to $Q_{\vartheta'' \vartheta} (t)$ is uniform on
$[0,T]$. Moreover, by (\ref{A10a}) we have that $Q_{\vartheta''
\vartheta}^{(n)} (0) = S_{\vartheta''\vartheta} (0)$, cf.
(\ref{A2a}), which yields the stated property of $Q_{\vartheta''
\vartheta}(0)$. Finally, (\ref{A10}) follows from (\ref{A7}) by the
same arguments.
\end{proof}
From (\ref{A10}) one can get that the family mentioned in
Proposition \ref{A2pn} enjoys the following `semigroup' property
\begin{equation}
  \label{A100}
 Q_{\vartheta' \vartheta} (t+s) = Q_{\vartheta' \vartheta''} (t) Q_{\vartheta'' \vartheta}
 (s),
\end{equation}
holding whenever $(\vartheta, \vartheta', t+s)$, $(\vartheta'',
\vartheta', t)$, and $(\vartheta, \vartheta'', s)$ are in
$\varUpsilon$.

Now we make precise which Cauchy problem we are going to solve. Set
\begin{equation}
  \label{A1a}
 \mathcal{D}_\vartheta = \{ k\in \mathcal{K}_\vartheta: L^\Delta k
 \in \mathcal{K}_\vartheta\},
\end{equation}
where $L^\Delta$ is as in (\ref{J10}). This defines an unbounded
linear operator $L^\Delta_\vartheta: \mathcal{D}_\vartheta \to
\mathcal{K}_\vartheta$, being the extension of the operators
$L^\Delta_{\vartheta'' \vartheta_0}: \mathcal{K}_{\vartheta_0} \to
\mathcal{K}_{\vartheta''} \hookrightarrow \mathcal{K}_\vartheta$
with $\vartheta'' \in (\vartheta_0, \vartheta)$ and all
$\vartheta_0< \vartheta$, cf. (\ref{19}). Then we consider the
Cauchy problem (\ref{J7}) in $\mathcal{K}_{\vartheta_1}$ with this
operator $L^\Delta_{\vartheta_1}$ and $k_{\mu_0}\in
\mathcal{K}_{\vartheta_0}$. By its (classical) solution we
understand the corresponding map $t \mapsto k_t \in
\mathcal{D}_{\vartheta_1}$, continuously differentiable in
$\mathcal{K}_{\vartheta_1}$.
\begin{lemma}
  \label{A1lm}
Let $\vartheta_0$ and $\vartheta_1$ be as in (\ref{A9}). Then for
each $k_{\mu_0} \in \mathcal{K}_{\vartheta_0}$, the problem
(\ref{J7}) as described above, cf (\ref{A1c}) below, has a unique
solution $k_t\in \mathcal{K}_{\vartheta_1}$ with $t \in [0,
T(\vartheta_1, \vartheta_0))$ given by the formula
\begin{equation}
  \label{A11}
 k_t = Q_{\vartheta_1 \vartheta_0}(t)k_{\mu_0},
\end{equation}
such that $k_t(\emptyset)=1$ for all $t \in [0, T(\vartheta_1,
\vartheta_0))$.
\end{lemma}
\begin{proof}
For each $t< T(\vartheta_1, \vartheta_0)$, one finds $\vartheta''
\in (\vartheta_0, \vartheta_1)$ such that also $t< T(\vartheta'',
\vartheta_0)$. By (\ref{A10}) we then get
\begin{equation}
 \label{A1c}
\frac{d}{dt} k_t = L^\Delta_{\vartheta_1 \vartheta''} k_t =
L^\Delta_{\vartheta_1}k_t.
\end{equation}
By claim (i) of Proposition \ref{A2pn} we have that $k_0 =
k_{\mu_0}$. Moreover, $k_t(\emptyset)=1$ for all $t \in [0,
T(\vartheta_1, \vartheta_0))$ since $k_0 = k_{\mu_0}$, see (b) in
(\ref{9h}), and
\[
\left(\frac{d}{dt} k_t \right)(\emptyset) = \left(
L^\Delta_{\vartheta_1 \vartheta''} k_t \right)(\emptyset) = 0,
\]
which follows from (\ref{J8a}) -- (\ref{J10}). The stated uniqueness
follows by the arguments used in the proof of Lemma 4.8 in
\cite{KK}.
\end{proof}
\begin{remark}
  \label{AArk}
As in the proof of Lemma \ref{A1lm} one can show that, for each
$t\in [0, T(\vartheta_1, \vartheta_0))$, the following holds:
$$Q_{\vartheta_1 \vartheta_0}(t): \mathcal{K}_{\vartheta_0} \to
\mathcal{D}_{\vartheta_1},$$ see (\ref{A1a}), and
\[
\qquad \frac{d}{dt} Q_{\vartheta_1 \vartheta_0}(t) =
L^\Delta_{\vartheta_1} Q_{\vartheta_1 \vartheta_0}(t).
\]
\end{remark}
Now we construct the evolution of functions $G_0 \mapsto G_t$ such
that, for $k\in \mathcal{K}_{\vartheta}$, the following holds, cf.
(\ref{J7a}),
\begin{equation}
  \label{A11a}
 \langle \!\langle G_0 ,Q_{\vartheta'\vartheta}(t) k \rangle \!
 \rangle = \langle \!\langle G_t , k \rangle \!
 \rangle, \qquad t< T(\vartheta', \vartheta).
\end{equation}
To this end, we introduce, cf (\ref{18}) and (\ref{17a}),
\begin{eqnarray}
  \label{A11b}
|G|_\vartheta & = & \int_{\Gamma_0} |G(\eta) | \exp\left( \vartheta
|\eta|\right) \lambda (d \eta), \\[.2cm] \mathcal{G}_\vartheta & = &
\{G:\Gamma_0\to \mathds{R}: |G|_\vartheta < \infty\}. \nonumber
\end{eqnarray}
Clearly, $\mathcal{G}_{\vartheta'} \hookrightarrow
\mathcal{G}_{\vartheta}$ for $\vartheta < \vartheta'$; hence, we
have introduced another scale of Banach spaces, cf (\ref{19}). As in
(\ref{A1}) and (\ref{A2a}), we define the corresponding
multiplication operators $A_{\vartheta\vartheta'}$ and
$S_{\vartheta\vartheta'}(t)$, and also $C_{\vartheta\vartheta'}\in
\mathcal{L}(\mathcal{G}_{\vartheta'} , \mathcal{G}_{\vartheta})$
which acts as
\begin{equation*}
 \left( CG\right)(\eta) = -\sum_{x\in \eta} \left(\sum_{y\in \eta \setminus
 x} a(x-y) \right) G(\eta \setminus x) + \int_{\mathds{R}^d} b(x) G(\eta \cup x) d
 x,
\end{equation*}
and thus satisfies, cf (\ref{A2}),
\begin{equation}
  \label{A11d}
 \|C_{\vartheta\vartheta'} \| \leq \frac{\|b\|e^{-\vartheta}+ \langle a \rangle e^{\vartheta'}}{e(\vartheta' -
 \vartheta)}.
\end{equation}
Now, for the same division of $[\vartheta, \vartheta']$ as in
(\ref{A5}), we introduce, cf (\ref{A6}),
\begin{gather*}
  \Omega^{l}_{\vartheta \vartheta'} (t, t_1, \dots , t_l) =
  S_{\vartheta \vartheta^1}(t_l) C_{\vartheta^1 \vartheta^2}
  S_{\vartheta^2\vartheta^3}(t_{l-1}- t_l) \cdots
  C_{\vartheta^{2s-1} \vartheta^{2s}}\\[.2cm]
  \times S_{\vartheta^{2s}\vartheta^{2s+1}}(t_{l-s}- t_{l-s+1})
  \cdots C_{\vartheta^{2l-1}\vartheta^{2l}}S_{\vartheta^{2l-1}\vartheta'}(t-
  t_{1}).
\end{gather*}
For this $\Omega^{l}_{\vartheta \vartheta'}$, one can get the
properties analogous to those stated in Proposition \ref{A1pn}.
Next, for $n\in \mathds{N}$, we define, cf (\ref{A10a}),
\begin{gather}
\label{AH} H^{(n)}_{\vartheta \vartheta'} (t) = S_{\vartheta
\vartheta'}(t) + \sum_{l=1}^n \int_0^t\int_0^{t_1} \cdots
\int_0^{t_{l-1}} \Omega^{l}_{\vartheta \vartheta'} (t, t_1, \dots ,
t_l) d t_l \cdots d t_1.
\end{gather}
As in the proof of Proposition \ref{A2pn}, by means of (\ref{A11d})
we then show that the sequence $\{H^{(n)}_{\vartheta \vartheta'}
(t)\}_{n\in \mathds{N}}$ converges in
$\mathcal{L}(\mathcal{G}_{\vartheta'} , \mathcal{G}_{\vartheta})$,
uniformly on compact subsets of $[0, T(\vartheta', \vartheta))$. Let
$H_{\vartheta \vartheta'} (t)$ be the limit. Then, by the very
construction in (\ref{AH}), it follows that, cf (\ref{A11a}),
\begin{equation}
  \label{AH1}
 \langle \! \langle H_{\vartheta \vartheta'} (t) G , k\rangle
 \!\rangle = \langle \! \langle  G , Q_{\vartheta' \vartheta} (t) k\rangle
 \!\rangle, \qquad t\in [0, T(\vartheta', \vartheta)),
\end{equation}
holding for each $G\in \mathcal{G}_{\vartheta'}$ and $k \in
\mathcal{K}_\vartheta$.

\subsection{The identification}
Our next step is based on the following statement.
\begin{lemma}
  \label{A2lm}
Let $\{Q_{\vartheta'\vartheta}(t): (\vartheta , \vartheta', t) \in
\varUpsilon\}$ be the family as in Proposition \ref{A2pn}. Then, for
each $\vartheta$ and $\vartheta'$ and $t\in [0, T(\vartheta',
\vartheta)/2)$, it follows that that $Q_{\vartheta'\vartheta} (t):
\mathcal{K}_{\vartheta}^\star\to \mathcal{K}_{\vartheta'}^\star$.
\end{lemma}
We prove this lemma in a number of steps. First we introduce
auxiliary models, indexed by $\sigma>0$, for which we construct the
families of operators $Q^\sigma_{\vartheta'\vartheta}(t)$ analogous
to those in Proposition \ref{A2pn}. Then we prove that these
families have the property stated in Lemma \ref{A2lm}. Thereafter,
we show that
\begin{equation}
  \label{A12}
\langle\!\langle G, Q^\sigma_{\vartheta_1\vartheta_0}(t) k_0
\rangle\!\rangle =: \langle\!\langle G, k^\sigma_t \rangle\!\rangle
\to \langle\!\langle G,
 k_t
 \rangle\!\rangle, \qquad {\rm as} \ \ \sigma \to 0,
\end{equation}
holding for each $G\in B^\star_{\rm bs}(\Gamma_0)$ and $k_t$ as in
Lemma \ref{A1lm} with $t\in [0,T(\vartheta',\vartheta)/2)$, see
(\ref{9g}). By Proposition \ref{Gpn} this yields the fact we wish to
prove.

\subsubsection{Auxiliary models}
For $\sigma>0$, we set
\begin{equation}
  \label{A13}
\varphi_\sigma (x) = \exp\left( - \sigma |x|^2\right), \qquad
b_\sigma (x) = b(x) \varphi_\sigma (x).
\end{equation}
Let also $L^{\Delta,\sigma}$ stand for $L^\Delta$ as in (\ref{J10})
with $b$ replaced by $b_\sigma$. Note that $\|b_\sigma\| \leq
\|b\|$. Clearly, for this $L^{\Delta,\sigma}$,  we can perform the
same construction as in the previous subsection and obtain the
family $\{Q^\sigma_{\vartheta'\vartheta}(t): (\vartheta ,
\vartheta', t) \in \varUpsilon\}$ as in Proposition \ref{A2pn} with
$\varUpsilon$ and $T(\vartheta', \vartheta)$ given in (\ref{A9}) and
(\ref{A8}), respectively. Note also that
$Q^\sigma_{\vartheta'\vartheta}(t)$ satisfy, cf. (\ref{A10}) and
Remark \ref{AArk},
\begin{equation}
 \label{A14A}
 \frac{d}{dt} Q^\sigma_{\vartheta'\vartheta}(t) = L^{\Delta, \sigma}_{\vartheta' \vartheta''} Q^\sigma_{\vartheta''\vartheta}(t)
 =L^{\Delta, \sigma}_{\vartheta'} Q^\sigma_{\vartheta'\vartheta}(t).
\end{equation}
Like in (\ref{A11}) we then set
\begin{equation}
  \label{A14}
 k^\sigma_t = Q^\sigma_{\vartheta_1\vartheta_0}(t) k_{\mu_0} , \qquad t
 < T(\vartheta_1 , \vartheta_0).
\end{equation}
Also as above, we construct the operators $H^\sigma_{\vartheta
\vartheta'}(t)$ such that, cf. (\ref{AH1}),
\begin{equation}
 \label{AH1a}
 \langle\! \langle H^\sigma_{\vartheta
\vartheta'}(t)G, k \rangle \! \rangle = \langle\! \langle G,
Q^\sigma_{\vartheta' \vartheta}(t)k  \rangle \! \rangle,
\end{equation}
holing for appropriate $G$ and $k$.
\begin{proposition}
  \label{AA1pn}
Assume that
$Q^\sigma_{\vartheta_1\vartheta_0}:\mathcal{K}^\star_{\vartheta_0}
\to \mathcal{K}^\star_{\vartheta_1}$ for all $t< T(\vartheta_1,
\vartheta_0)$. Then,  for all $t < T(\vartheta_1, \vartheta_0)/2$
and $G\in B_{\rm bs}(\Gamma_0)$, the convergence in (\ref{A12})
holds.
\end{proposition}
\begin{proof}
Take $\vartheta = (\vartheta_1 + \vartheta_0)/2$ and then pick
$\vartheta' \in (\vartheta, \vartheta_1)$ such that
\begin{equation}
  \label{A15}
  \frac{1}{2} T(\vartheta_1 , \vartheta_0)  \leq \min\{T(\vartheta_1, \vartheta'); T(\vartheta,
  \vartheta_0)\},
\end{equation}
which is possible in view of the continuous dependence of
$T(\vartheta', \vartheta)$ on both its arguments, see (\ref{A8}).
For $t< T(\vartheta_1 , \vartheta_0)/2$, by (\ref{A11}) and
(\ref{A14}) we get that
\begin{gather}
  \label{A16}
k_t - k^\sigma_t = \int_0^t Q_{\vartheta_1 \vartheta'}(t-s) \left(
L^{\Delta}_{\vartheta'\vartheta}-
L^{\Delta,\sigma}_{\vartheta'\vartheta} \right) k^\sigma_s d s =:
\int_0^t Q_{\vartheta_1 \vartheta'}(t-s) D_{\vartheta'\vartheta}
k^\sigma_s d s,
\end{gather}
where (\ref{J10}),
\begin{equation}
  \label{A17}
(D k)(\eta)= \sum_{x\in \eta} \left( 1 - \varphi_\sigma(x)\right)
b(x) k(\eta\setminus x),
\end{equation}
see (\ref{J10}) and $k^\sigma_s$ lies in $\mathcal{K}_\vartheta$,
which is possible since $$s \leq t < \frac{1}{2}T(\vartheta_1,
\vartheta_0) \leq T(\vartheta, \vartheta_0), $$ see (\ref{A15}).
Take $G\in B_{\rm bs}$. Since it lies in each
$\mathcal{G}_\vartheta$, and hence in $\mathcal{G}_{\vartheta_1}$,
we can get $$H_{\vartheta'\vartheta_1}(t-s) G =:G_{t-s} \in
\mathcal{G}_{\vartheta'}, \qquad t-s < T(\vartheta_1,
\vartheta_0)/2,$$ see (\ref{A15}).
 For this $G$, by (\ref{AH1}) and (\ref{A16}) we have
\begin{gather}
  \label{A18}
\psi_\sigma (t) := \langle \!\langle G, k_t - k^\sigma_t \rangle \!
\rangle = \int_0^t \langle\!\langle G_{t-s}, D_{\vartheta'
\vartheta}k^\sigma_s \rangle \! \rangle ds \\[.2cm]
= \int_0^t \left( \int_{\Gamma_0}\int_{\mathds{R}^d}
\frac{1}{|\eta|+1} G_{t-s}(\eta\cup x) b(x) (1 - \varphi_\sigma (x))
(|\eta|+1)k^\sigma_s (\eta) d x \lambda (d \eta) \right) ds.
\nonumber
\end{gather}
To get the latter line we also used (\ref{A17}). Recall that here
$G_{t-s}\in \mathcal{G}_{\vartheta'}$ and $k_s^\sigma \in
\mathcal{K}_\vartheta$ with $\vartheta < \vartheta'$. Let us prove
that
\begin{equation*}
g_s (x) := \int_{\Gamma_0} \frac{1}{|\eta|+1}\left\vert G_s
(\eta\cup x)\right\vert \exp\left( \vartheta'|\eta|\right) \lambda
(d \eta)
\end{equation*}
lies in $L^1 (\mathds{R}^d)$ for each $s\leq \widetilde{T}$. Indeed,
by (\ref{8}) and (\ref{A11b}) we have
\begin{equation}
  \label{A20}
\|g_s \|_{L^1 (\mathds{R}^d)} \leq e^{-\vartheta'}\sup_{s\in
[0,\widetilde{T}]} |G_s|_{\vartheta'}.
\end{equation}
We use this in (\ref{A18}) to get
\begin{gather*}
|\psi_\sigma(t) | \leq \sup_{s\in [0,\widetilde{T}]}
\|k_s\|_\vartheta \frac{\| b \|e^{\vartheta' - \vartheta-1}
}{\vartheta'-\vartheta}\int_0^t \int_{\mathds{R}^d} g_s (x) \left(1
- \varphi_\sigma (x)\right) d x ds \to 0, \qquad {\rm as}  \ \
\sigma \to 0.
\end{gather*}
The latter convergence follows by (\ref{A20}) and the Lebesgue
dominated convergence theorem. This completes the proof.
\end{proof}

\subsubsection{Auxiliary evolutions}
Now we turn to proving that the assumption of Proposition
\ref{AA1pn} holds true. For a compact $\Lambda$, by $\Gamma_\Lambda$
we denote the set of configurations $\eta$ contained in $\Lambda$.
It is a measurable subset of $\Gamma_0$, i.e., $\Gamma_\Lambda\in
\mathcal{B}(\Gamma)$. Recall that $\mathcal{B}(\Gamma)$ can be
generated by the cylinder sets $\Gamma^{\Lambda,n}$ with all
possible compact $\Lambda$ and $n\in \mathds{N}_0$. Let
$\mathcal{B}(\Gamma_\Lambda)$ denote the sub-$\sigma$-field of
$\mathcal{B}(\Gamma)$ consisting of  $A\subset \Gamma_\Lambda$. For
$A\in\mathcal{B}(\Gamma_\Lambda)$, we set $C_\Lambda(A) = \{ \gamma
\in \Gamma: \gamma_\Lambda \in A\}$. Then, for a state $\mu$, we
define $\mu^\Lambda$ by setting $\mu^\Lambda(A) =
\mu(C_\Lambda(A))$; thereby, $\mu^\Lambda$ is a probability measure
on $\mathcal{B}(\Gamma_\Lambda)$.
 It is possible
to show, see \cite{Tobi},  that for each compact $\Lambda$ and $\mu
\in \mathcal{P}_{\rm exp} (\Gamma)$, the measure $\mu^\Lambda$ has
density with respect to the Lebesgue-Poisson measure defined in
(\ref{8}), which we denote by $R_\mu^\Lambda$. Moreover, the
correlation function $k_\mu$ and the density $R_\mu^\Lambda$ satisfy
\begin{equation}
  \label{A21}
k_\mu (\eta) = \int_{\Gamma_\Lambda} R^\Lambda_\mu (\eta \cup \xi)
\lambda (d \xi), \qquad \eta \in \Gamma_\Lambda.
\end{equation}
Let $\mu_0\in \mathcal{P}_{\rm exp}(\Gamma)$ be the initial state as
in Lemma \ref{A1lm}. Fix some compact $\Lambda$ and $N\in
\mathds{N}$, and then, for $\eta\in \Gamma_0$, set
\begin{equation}
  \label{A22}
R^{\Lambda,N}_0 (\eta) = \left\{ \begin{array}{ll}
R^{\Lambda}_{\mu_0} (\eta), \qquad  &{\rm if} \ \ \eta \subset
\Lambda \ \ {\rm and} \ \ |\eta|\leq N;\\[.3cm] 0, \qquad  &{\rm
otherwise.}
\end{array}\right.
\end{equation}
Clearly, $R^{\Lambda,N}_0\in \mathcal{G}_\vartheta$ with any
$\vartheta \in \mathds{R}$, and $R^{\Lambda,N}_0(\eta) \geq 0$ for
$\lambda$-almost all $\eta\in \Gamma_0$.

Let us now consider the auxiliary model specified by $L^{\Delta,
\sigma}$, and also by $L^\sigma$ which one obtains by replacing in
(\ref{L}) $b$ by $b_\sigma$, see (\ref{A13}). Then the equation for
the densities obtained from the Fokker-Planck equation (\ref{R1})
takes the form
\begin{eqnarray}
  \label{A25}
 \frac{d}{dt} R_t (\eta) & = &(L^\dagger R_t )(\eta)\\[.2cm] \nonumber & := & - \Psi_\sigma (\eta) R_t(\eta)
 + \sum_{x\in \eta} b_\sigma(x) R_t (\eta \setminus x) +
 \int_{\mathds{R}^d} \left( \sum_{y\in \eta} a(x-y)\right) R_t(\eta \cup x) dx,
\end{eqnarray}
where
\begin{equation}
  \label{A26}
 \Psi_\sigma (\eta) := E(\eta) + \langle b_\sigma \rangle,
 \qquad \langle b_\sigma \rangle := \int_{\mathds{R}^d} b(x)
 \varphi_\sigma(x) d x.
\end{equation}
Set
\begin{equation*}
 \mathcal{G}^{+}_\vartheta = \{ G \in \mathcal{G}_\vartheta: G(\eta)
 \geq 0 , \ \ {\rm for} \ \lambda-{\rm a.a.} \ \eta\in \Gamma_0\},
\end{equation*}
and also
\begin{equation}
  \label{A28}
\mathcal{D}= \{ R\in \mathcal{G}_0 : \Psi_\sigma  R \in
\mathcal{G}_0 \}, \qquad \mathcal{D}^+= \mathcal{D} \bigcap
\mathcal{G}_0^+.
\end{equation}
\begin{proposition}
  \label{AP1pn}
The operator $(L^\dagger, \mathcal{D})$ defined in (\ref{A25}) and
(\ref{A28}) is the generator of a substochastic semigroup $S^\dagger
= \{S^\dagger (t)\}_{t \geq 0}$ on $\mathcal{G}_0$, which leaves
invariant each $\mathcal{G}_\vartheta$, $\vartheta > 0$.
\end{proposition}
\begin{proof}
In this statement we mean that
\begin{eqnarray}
  \label{A28a}
\forall t\geq 0 \qquad &(a) & \ \ S^\dagger(t) : \mathcal{G}^+_0 \to
\mathcal{D}^+; \\[.2cm]  & (b) & \ \ |S^\dagger(t) R|_0 \leq 1, \ \ {\rm
whenever} \ \ |R|_0 \leq 1 \ \ {\rm and} \ \ R\in \mathcal{G}^+_0; \nonumber \\[.2cm]  & (c) & \ \ S^\dagger(t) : \mathcal{G}^+_\vartheta
\to \mathcal{G}^+_\vartheta, \ \ {\rm for} \ \ {\rm all} \ \
\vartheta>0. \nonumber
\end{eqnarray}
We use the Thieme-Voigt theorem in the form of \cite[Propositions
3.1 and 3.2]{K}. By this theorem the proof amounts to checking the
validity of the following inequalities:
\begin{gather}
  \label{A29}
\forall R\in \mathcal{D}^+ \qquad \int_{\Gamma_0} \left(L^\dagger R
\right)(\eta) \lambda ( d\eta) \leq 0,\\[.2cm] \nonumber \forall \vartheta >0 \qquad
(L^\sigma G_\vartheta)(\eta) + \varepsilon \Psi_\sigma (\eta) \leq C
G_\vartheta (\eta), \qquad G_\vartheta (\eta) :=
e^{\vartheta|\eta|},
\end{gather}
holding for some positive $C$ and $\varepsilon$. Recall that
$\Psi_\sigma$ is defined in (\ref{A26}). By direct inspection  we
get from (\ref{A25}) that the left-hand side of the first line in
(\ref{A29}) equals zero for each $R\in \mathcal{D}$. Proving the
second inequality in (\ref{A29}) reduces to showing that, for each
$\vartheta>0$, the function
\[
\varSigma (\eta) := - E (\eta) (1 - e^{-\vartheta} ) +\langle
b_\sigma \rangle (e^\vartheta-1) + \varepsilon \Psi_\sigma (\eta)
e^{-\vartheta|\eta|}, \qquad \eta \in \Gamma_0,
\]
is bounded from above, which is obviously the case.
\end{proof}
The second auxiliary evolution is supposed to be constructed in
$\mathcal{G}_\vartheta$. It is generated by the operator
$\widehat{L}_\vartheta$ the action of which coincides with that of
$L^{\Delta, \sigma}$, see (\ref{J10}) and (\ref{J8a}), with $b$
replaced by $b_\sigma$. The domain of this operator is
\begin{equation}
  \label{A30}
\widehat{\mathcal{D}}_\vartheta = \{ q \in \mathcal{G}_\vartheta:
\Psi_\sigma (\cdot ) q \in \mathcal{G}_\vartheta\}.
\end{equation}
\begin{proposition}
  \label{AP2pn}
For each $\vartheta >0$, the operator
$(\widehat{L}_\vartheta,\widehat{\mathcal{D}}_\vartheta)$ is the
generator of a $C_0$-semigroup $\widehat{S}_\vartheta:=
\{\widehat{S}_\vartheta(t)\}_{t\geq 0}$ of bounded operators on
$\mathcal{G}_\vartheta$.
\end{proposition}
\begin{proof}
As in the proof of Lemma 5.5 in \cite{KK}, we pass from $q$ to $w$
by setting $w(\eta) = (-1)^{|\eta|} q(\eta)$, and hence to
$\widetilde{L}_\vartheta$ defined on the same domain (\ref{A30}) by
the relation $(\widetilde{L}_\vartheta w)(\eta) =
(-1)^{|\eta|}(\widehat{L}_\vartheta  q)(\eta)$. Then we just prove
that $(\widetilde{L}_\vartheta, \widehat{\mathcal{D}}_\vartheta)$
generates a $C_0$-semigroup on $\mathcal{G}_\vartheta$. In view of
(\ref{J8a}) -- (\ref{J10}), we have
\begin{eqnarray*}
\widetilde{L}_\vartheta & = & \widetilde{A}  + \widetilde{B}  +
\widetilde{C} \\[.2cm] \nonumber
(\widetilde{A} w)(\eta) & = & - E (\eta) w(\eta), \quad
(\widetilde{B} w)(\eta) = \int_{\mathds{R}^d} \left(\sum_{y\in \eta}
a(x-y) \right)
w(\eta\cup x) d x ,\\[.2cm] \nonumber
(\widetilde{C} w)(\eta) & = & - \sum_{x\in \eta} b_\sigma (x)
w(\eta\setminus x).
\end{eqnarray*}
By (\ref{A11b}) we get
\[
|\widetilde{C} w|_\vartheta \leq e^\vartheta \langle b_\sigma
\rangle  |w|_\vartheta,
\]
hence $\widetilde{C}$ is a bounded operator. For $w\in
\mathcal{G}_\vartheta^{+}$, we have
\begin{gather*}
  |\widetilde{B} w|_\vartheta = \int_{\Gamma_0}e^{\vartheta|\eta|}\left( \int_{\mathds{R}^d} \left( \sum_{y\in \eta} a (x-y) \right)
w(\eta\cup x) d x \right) \lambda ( d\eta) \\[.2cm] =
\int_{\Gamma_0}e^{\vartheta(|\eta|-1)} \left(\sum_{x\in
\eta}\sum_{y\in \eta\setminus x} a(x-y) \right) w(\eta) \lambda
(d\eta) \\[.2cm]
\leq e^{-\vartheta} \int_{\Gamma_0}e^{\vartheta|\eta|} E (\eta)
w(\eta) \lambda ( d\eta) = e^{-\vartheta} |\widetilde{A}
w|_\vartheta < |\widetilde{A} w|_\vartheta.
\end{gather*}
The latter estimate allows us to apply here the Thieme-Voigt
theorem, see \cite[Proposition 3.1]{K} by which $\widetilde{A}
+\widetilde{B}$ generates a substochastic semigroup in
$\mathcal{G}_\vartheta$. Thus, $\widetilde{L}_\vartheta$ generates a
$C_0$-semigroup since $\widetilde{C}$ is bounded. This completes the
proof.
\end{proof}
Now for $R_0^{\Lambda, N}$ defined in (\ref{A22}),
 we set
\begin{equation}
  \label{A23}
q^{\Lambda,N}_0 (\eta) = \int_{\Gamma_0} R^{\Lambda,N}_0 (\eta\cup
\xi) \lambda (d \xi), \qquad \eta \in \Gamma_0.
\end{equation}
By (\ref{A21})
\begin{equation}
  \label{A24}
0 \leq q^{\Lambda,N}_0 (\eta) \leq k_{\mu_0} (\eta).
\end{equation}
Hence, $q^{\Lambda,N}_0\in \mathcal{K}_{\vartheta_0}$. By
(\ref{A22}) $R_0^{\Lambda,N}$ lies in each $\mathcal{G}_\vartheta$,
$\vartheta \geq 0$. At the same time,
\begin{eqnarray*}
|q_0^{\Lambda,N}|_{\vartheta} & = & \int_{\Gamma_0}\int_{\Gamma_0}
e^{\vartheta|\eta|} R_0^{\Lambda,N}(\eta\cup \xi) \lambda (d\eta)
\lambda (d \xi) \\[.2cm] \nonumber & = & \int_{\Gamma_0}\left(
\sum_{\eta \subset \xi} e^{\vartheta |\eta|}\right)
R_0^{\Lambda,N}(\xi) \lambda (d\xi) = |R_0^{\Lambda,N}|_\beta,
\end{eqnarray*}
where $\beta>0$ is to satisfy $e^\beta = 1 + e^\vartheta$. Hence,
$q_0^{\Lambda,N}\in \mathcal{G}_\vartheta$ for each $\vartheta>0$.
In view of this, $q_0^{\Lambda,N}\in
\widehat{\mathcal{D}}_\vartheta$ for each $\vartheta>0$, see
(\ref{A30}). Consider the problem in $\mathcal{G}_\vartheta$
\begin{equation}
  \label{A33}
\frac{d}{dt} q_t = \widehat{L}_\vartheta q_t, \qquad q_t|_{t=0} =
q_0^{\Lambda,N}.
\end{equation}
\begin{proposition}
  \label{B1pn}
For each $\vartheta>0$, the problem in (\ref{A33}) has a unique
global solution $q_t \in \widehat{\mathcal{D}}_\vartheta$ such that,
for each $G\in B_{\rm bs}^\star(\Gamma_0)$, the following holds
\begin{equation}
  \label{A34}
\langle\! \langle G, q_t \rangle \!\rangle \geq 0.
\end{equation}
\end{proposition}
\begin{proof}
By Proposition \ref{AP2pn} the problem in (\ref{A33}) has a unique
global solution given by
\begin{equation}
  \label{A35}
q_t = \widehat{S}_\vartheta (t) q^{\Lambda,N}_0.
\end{equation}
On the other hand, this solution can be sought in the form
\begin{equation}
  \label{A36}
q_t(\eta) = \int_{\Gamma_0} \left(S^\dagger(t) R_0^{\Lambda,N}
\right)(\eta \cup \xi) \lambda (d \xi),
\end{equation}
where $S^\dagger$ is the semigroup constructed in Proposition
\ref{AP1pn}. Indeed, by direct inspection one verifies that $q_t$ in
this form satisfies (\ref{A33}), cf the proof of Lemma 5.8 in
\cite{KK}. Then, cf (\ref{9a}),
\begin{gather}
\label{A37}
 \langle\! \langle G, q_t \rangle \!\rangle =
\int_{\Gamma_0} (KG)(\eta) \left(S^\dagger(t) R_0^{\Lambda,N}
\right)(\eta) \lambda(d\eta) \geq 0,
\end{gather}
which yields (\ref{A34}). The inequality in (\ref{A37}) follows by
the fact that the semigroup $S^\dagger$ is substochastic, see
(\ref{A28a}). This completes the proof.
\end{proof}
By (\ref{A24}) it follows that $q^{\Lambda,N}_0 \in
\mathcal{K}_{\vartheta_0}$, hence we may use it in (\ref{A14}) and
obtain
\begin{equation}
 \label{A38}
k_t^{\Lambda,N} = Q^\sigma_{\vartheta_1
\vartheta_0}(t)q^{\Lambda,N}_0, \qquad t \in [0,T(\vartheta_1,
\vartheta_0)).
\end{equation}
\begin{proposition}
 \label{B2pn}
Let $k_t^{\Lambda,N}$ and $q_t$ be as in (\ref{A38}) and in
(\ref{A35}), (\ref{A36}), respectively. Then, for all $t \in
[0,T(\vartheta_1, \vartheta_0))$, it follows that
$k_t^{\Lambda,N}=q_t$.
\end{proposition}
\begin{proof}
A priori $k_t^{\Lambda,N}$ and $q_t$ lie in different spaces:
$\mathcal{K}_{\vartheta_1}$ and $\mathcal{G}_\vartheta$,
respectively. Note that the latter $\vartheta$ can be arbitrary
positive. The idea is to construct one more evolution
$q^{\Lambda,N}_0\mapsto u_t$ in some intersection of these two
spaces, related to the evolutions in (\ref{A38}) and (\ref{A35}).
Then the proof will follow by the uniqueness as in Proposition
\ref{B1pn}.

For $\vartheta \in \mathds{R}$, $\varphi_\sigma$ as in (\ref{A13})
and $u:\Gamma_0 \to \mathds{R}$, we set, cf (\ref{17a}) and
(\ref{1fa}),
\begin{equation}
 \label{A39}
 \|u\|_{\sigma, \vartheta} = \esssup_{\eta \in \Gamma_0} \frac{|u(\eta)|\exp\left(- \vartheta |\eta|\right)}{e(\eta;\varphi_\sigma)}
 , \qquad  e(\eta;\varphi_\sigma):=\prod_{x\in \eta} \varphi_\sigma(x),
\end{equation}
and then $\mathcal{U}_{\sigma, \vartheta}:= \{ u:\Gamma_0 \to
\mathds{R}: \|u\|_{\sigma, \vartheta} < \infty\}$. Clearly,
\begin{equation}
 \label{A40}
 \mathcal{U}_{\sigma, \vartheta} \hookrightarrow \mathcal{K}_{\vartheta}, \qquad \vartheta \in \mathds{R},
\end{equation}
since $\|u\|_\vartheta \leq \|u\|_{\sigma, \vartheta}$. Moreover, as
in (\ref{19}) we have that $ \mathcal{U}_{\sigma, \vartheta}
\hookrightarrow \mathcal{U}_{\sigma, \vartheta'}$ for $\vartheta' >
\vartheta$. Let $L^{\Delta, \sigma}$ be defined as in (\ref{J10})
with $b$ replaced by $b_\sigma$. Then we define an unbounded linear
operator $L^{\Delta, \sigma}_{\vartheta,u}: \mathcal{D}^{\Delta,
\sigma}_{\vartheta,u} \to \mathcal{U}_{\sigma, \vartheta}$ with the
action as just described and the domain
\begin{equation}
 \label{A41}
 \mathcal{D}^{\Delta, \sigma}_{\vartheta,u}=\{ u\in \mathcal{U}_{\sigma, \vartheta} : L^{\Delta, \sigma} u \in \mathcal{U}_{\sigma, \vartheta}\}.
\end{equation}
Clearly, $\mathcal{U}_{\sigma, \vartheta''} \subset
\mathcal{D}^{\Delta, \sigma}_{\vartheta,u}$ for each $\vartheta'' <
\vartheta$. By (\ref{A22}) and (\ref{A23}) it follows that
$q^{\Lambda,N}_0 (\eta) = 0$ if $|\eta|>N$ or if $\eta$ is not
contained in $\Lambda$. Then $q^{\Lambda,N}_0$ lies in each
$\mathcal{U}_{\sigma, \vartheta''}$, $\vartheta''\in \mathds{R}$,
and hence in the domain of $L^{\Delta, \sigma}_{\vartheta,u}$ given
in (\ref{A41}). Thus, we can consider
\begin{equation}
 \label{A42}
 \frac{d}{dt} u_t = L^{\Delta, \sigma}_{\vartheta,u} u_t, \qquad u_t|_{t=0} = q^{\Lambda,N}_0.
\end{equation}
As in (\ref{A1}) we write $L^{\Delta, \sigma}_{\vartheta,u} =
A^{\sigma,u}+B^{\sigma,u}$, where $A^{\sigma,u}$ is the
multiplication operator by $-E(\eta)$. The operator norm of
$B^{\sigma,u}$ can be estimated as follows. By (\ref{A39}) we have
\[
 |u(\eta)| \leq \|u\|_{\sigma, \vartheta} \exp\left( \vartheta|\eta|\right)\prod_{x\in \eta}\varphi_\sigma (x),
\]
which yields
\begin{gather*}
 \left\vert \left(B^{\sigma,u}u\right)(\eta) \right\vert \leq
 \|u\|_{\sigma, \vartheta} |\eta| \exp\left( \vartheta|\eta|\right)
\left(\|b\|e^{-\vartheta} + \langle a \rangle e^\vartheta \right)
\prod_{x\in \eta}\varphi_\sigma (x).
 \end{gather*}
Hence, the operator norm of  $B^{\sigma,u}_{\vartheta'\vartheta}\in
\mathcal{L}( \mathcal{U}_{\sigma, \vartheta}, \mathcal{U}_{\sigma,
\vartheta'})$ satisfies
\begin{equation*}
 \|B^{\sigma,u}_{\vartheta',\vartheta}\| \leq \frac{\|b\|e^{-\vartheta} + \langle a \rangle e^{\vartheta'}}{e(\vartheta' - \vartheta)},
\end{equation*}
which coincides with that in (\ref{A2}). Then we repeat the
construction made in Propositions \ref{A1pn}, \ref{A2pn} and Lemma
\ref{A1lm} and obtain the solution of (\ref{A42}) in the form
\begin{equation*}
u_t = Q^{\sigma,u}_{\vartheta_1 \vartheta_0}(t)  q^{\Lambda,N}_0,
\qquad t\in [0, T(\vartheta_1, \vartheta_0)),
\end{equation*}
where $T(\vartheta_1, \vartheta_0)$ is as in (\ref{A8}) whereas
$Q^{\sigma,u}_{\vartheta_1 \vartheta_0}(t)$ satisfies, cf
(\ref{A10}) and Remark \ref{AArk},
\[
 \frac{d}{dt}  Q^{\sigma,u}_{\vartheta_1 \vartheta_0}(t) =
 \left(A^{\sigma,u}_{\vartheta_1 \vartheta'} +B^{\sigma,u}_{\vartheta_1 \vartheta'} \right) Q^{\sigma,u}_{\vartheta' \vartheta_0}(t)
= L^{\Delta, \sigma}_{\vartheta_1,u}  Q^{\sigma,u}_{\vartheta_1
\vartheta_0}(t).
 \]
Since $(L^{\Delta, \sigma}_{\vartheta_1,u}, \mathcal{D}^{\Delta,
\sigma}_{\vartheta_1,u}) \subset (L^{\Delta, \sigma}_{\vartheta_1},
\mathcal{D}^{\Delta, \sigma}_{\vartheta_1})$, and in view of
(\ref{A14A}) and (\ref{A38}), (\ref{A40}), we have that
\begin{equation}
  \label{A47}
\forall t\in [0, T(\vartheta_1, \vartheta_0)) \qquad k^\sigma_t =
u_t.
\end{equation}
On the other hand, for $\vartheta>0$ and $u\in \mathcal{U}_{\sigma,
\vartheta'}$, by (\ref{A39}) we get
\begin{gather*}
\int_{\Gamma_0}|u(\eta)| e^{\vartheta|\eta|} \lambda (d \eta) \leq
\|u\|_{\sigma, \vartheta'} \int_{\Gamma_0} \exp\left( (\vartheta' +
\vartheta)|\eta|\right) e(\eta; \varphi_\sigma) \lambda ( d \eta)
\\[ .2cm] \nonumber = \|u\|_{\sigma, \vartheta'} \exp\left( \langle \varphi_\sigma \rangle e^{\vartheta +
\vartheta'}\right), \qquad \langle \varphi_\sigma \rangle :=
\int_{\mathds{R}^d} \varphi_\sigma (x) d x.
\end{gather*}
Thus, $\mathcal{U}_{\sigma, \vartheta'}\hookrightarrow \mathcal{G}_{
\vartheta}$ for each $\vartheta'\in \mathds{R}$ and $\vartheta\geq
0$. Likewise, one shows that $\mathcal{D}^{\Delta,
\sigma}_{\vartheta',u}\hookrightarrow
\widehat{\mathcal{D}}_\vartheta$, see (\ref{A30}). Since the action
of $\widehat{L}$ coincides with that of $L^{\Delta,\sigma}$, by the
latter embedding we have that $(L^{\Delta, \sigma}_{\vartheta_1,u},
\mathcal{D}^{\Delta, \sigma}_{\vartheta_1,u}) \subset
(\widehat{L}_{\vartheta}, \widehat{\mathcal{D}}_{\vartheta})$,
holding for each $\vartheta>0$. Then by the uniqueness stated in
Proposition \ref{B1pn} we conclude that $q_t = u_t$ for all $t \in
[0,T(\vartheta_1, \vartheta_0))$. In view of (\ref{A47}), this
yields $k_t^{\Lambda,N} = u_t$, which completes the proof.
\end{proof}
\subsubsection{Proof of Lemma \ref{A2lm}} We have to show that the
assumption of Proposition \ref{AA1pn} holds true for each $\sigma
>0$, which is equivalent to proving that $k^\sigma_t$ given in
(\ref{A14}) has the property
\begin{equation}
  \label{A49}
\langle \! \langle G_0, k_t^\sigma \rangle \! \rangle \geq 0,
\end{equation}
holding for all $t < T(\vartheta_1, \vartheta_0)$ and $G_0\in
B^{\star}_{\rm bs}(\Gamma_0)$. By definition, a \emph{cofinal}
sequence of $\{\Lambda_n\}_{n\in \mathds{N}}$ is a sequence of
compact subsets $\Lambda_n \subset \mathds{R}^d$ such that
$\Lambda_n \subset \Lambda_{n+1}$, $n\in \mathds{N}$, and each $x
\in \mathds{R}^d$ is contained in a certain $\Lambda_n$. Let
$\{\Lambda_n\}_{n\in \mathds{N}}$ be such a sequence. Fix $\sigma>0$
and then, for given $\Lambda_n$ and $N\in \mathds{N}$, obtain
$q_0^{\Lambda_n,N}$ from $k_{\mu_0}\in \mathcal{K}_{\vartheta_0}$ by
(\ref{A22}), (\ref{A23}). As in \cite[Appendix]{Berns} one can show
that, for each $G\in \mathcal{G}_{\vartheta_0}$, the following holds
\begin{equation}
  \label{A50}
\lim_{n\to + \infty} \lim_{N\to + \infty} \langle \! \langle G,
q_0^{\Lambda_n,N} \rangle \! \rangle =  \langle \! \langle G,
k_{\mu_0} \rangle \! \rangle.
\end{equation}
Let $G_0$ be  as in (\ref{A49}) and hence lie in any
$\mathcal{G}_\vartheta$.  For $t\in [0,T(\vartheta_1, \vartheta_0))$
and $k_t^{\Lambda_n,N}$ as in (\ref{A38}), by (\ref{AH1a}) we get
\begin{gather}
  \label{A51}
\langle \! \langle G_0, k_t^{\Lambda_n,N} \rangle \! \rangle =
\langle \! \langle H^\sigma_{\vartheta_0\vartheta_1}(t) G_0,
q_0^{\Lambda_n,N} \rangle \! \rangle =: \langle \! \langle G,
q_0^{\Lambda_n,N} \rangle \! \rangle \geq 0.
\end{gather}
The latter inequality follows by Proposition \ref{B2pn} and
(\ref{A37}). Then (\ref{A49}) follows by (\ref{A50}) and
(\ref{A51}).

\subsection{Proof of Theorem \ref{1tm}}

To complete proving the theorem we have to construct the
continuation of the solution (\ref{A11}) to all $t\geq 0$ and prove
the upper bound in (\ref{T2}). The lower bound follows by the fact
that $k_t \in \mathcal{K}^\star$. This will be done by comparing
$k_t$ with the solution of the equation (\ref{J7}) for the Surgailis
model given in (\ref{T5}). If we denote the latter by $v_t$, then
\begin{gather}
  \label{A53}
v_t (\eta) = (W(t) k_{\mu_0})(\eta) := \sum_{\xi\subset \eta}e(\xi;
\phi_t) e(\eta \setminus \xi; \psi_t) k_{\mu_0}(\eta\setminus \xi).
\end{gather}
For $k_{\mu_0}\in\mathcal{K}_{\vartheta_0}$, by (\ref{17a}) and
(\ref{J6a}) we get from the latter
\begin{equation}
  \label{A54}
v_t (\eta) \leq \|k_{\mu_0}\|_{\vartheta_0} \exp\left\{\vartheta_0 +
\log\left(1 + t \|b\| e^{-\vartheta_0} \right) \right\},
\end{equation}
which holds also in the case $m\equiv 0$. Thus, for  a given $T>0$,
$W(t)$ with $t\in [0,T]$ acts as a bounded operator $W_{\vartheta_T
\vartheta_0}(t)$ from $\mathcal{K}_{\vartheta_0}$ to
$\mathcal{K}_{\vartheta_T}$ with
\begin{equation}
 \label{A54a}
 \vartheta_T := \vartheta_0 + \log \left(1 + T \|b\| e^{-\vartheta_0} \right).
\end{equation}
For $\vartheta\in \mathds{R}$, we set, cf (\ref{A8}),
\begin{equation}
  \label{A55}
  \tau(\vartheta) = T(\vartheta+1 , \vartheta) = \left[\|b \|e^{-\vartheta} + e \langle a \rangle e^\vartheta \right]^{-1}.
\end{equation}
For $\vartheta_1 = \vartheta_0 + 1$, let $k_t$ be given in
(\ref{A11}) with $t \in [0, \tau(\vartheta_0))$. Fix some $\kappa\in
(0,1/2)$ and set $T_1 = \kappa \tau(\vartheta_0)$. By Lemmas
\ref{A1lm} and \ref{A2lm} we know that $k_t = Q_{\vartheta_1
\vartheta_0} (t) k_{\mu_0}$ exists and lies in
$\mathcal{K}_{\vartheta_1}^\star$ for all $t\in [0,T_1]$. Take
$\vartheta \in (\vartheta_0 , \vartheta_0+1)$ such that $T_1 <
T(\vartheta, \vartheta_0)$. Then take $\vartheta'>\vartheta$ and
set, cf (\ref{A54a}),
\begin{equation*}
 \tilde{\vartheta}_1 = \max\left\{\vartheta_0 + 1; \vartheta' + \log \left( 1 + T_1 \|b\|e^{-\vartheta'}\right)\right\}.
\end{equation*}
For $t\in [0,T_1]$, we have
\begin{gather}
 \label{A56}
 v_t - k_t = \int_0^t W_{\tilde{\vartheta}_1 \vartheta'}(t-s) D_{\vartheta' \vartheta} k_s ds ,
\end{gather}
where $k_s$ belongs to $\mathcal{K}_\vartheta$, whereas $v_t$ and
$k_t$ belong to $\mathcal{K}_{\tilde{\vartheta}_1}$. By (\ref{A53})
and (\ref{A1}) the action of $D$ in (\ref{A56}) is
\[
 (D k)(\eta) = \left(\sum_{x\in \eta}\sum_{y\in \eta \setminus x} a(x-y)\right) k(\eta)
 + \int_{\mathds{R}^d} \left(\sum_{y\in \eta \setminus x} a(x-y)\right) k(\eta\cup x) d x,
\]
hence, $v_t (\eta) - k_t (\eta) \geq 0$ for $\lambda$-almost all
$\eta\in \Gamma_0$ since $W(t)$ is positive, see (\ref{A53}) and
(\ref{T3}), and $k_s \in \mathcal{K}_\vartheta^\star \subset
\mathcal{K}_\vartheta^+$, see (\ref{19a}), (\ref{19b}), and Lemma
\ref{A2lm}. Since  $k_t$ in (\ref{A56}) is in $\mathcal{K}^\star$,
we have that
\begin{equation}
  \label{T10}
 0 \leq k_t \leq v_t, \qquad t \in [0, T_1],
\end{equation}
which by (\ref{A54}) yields $k_t \in \mathcal{K}_{\vartheta_{T_1}}$
and the bound in (\ref{T2}) for such $t$, see (\ref{A54a}). Set $T_2
= \kappa \tau(\vartheta_{T_1})$, $\vartheta_2 = \vartheta_{T_1} +1$
and consider $k^{(2)}_t = Q_{\vartheta_2\vartheta_{T_1}}(t) k_{T_1}$
with $t\in [0,T_2]$. Clearly, $k^{(2)}_t = k_{T_1 + t}$ for $T_1 +
t< T(\vartheta_0+1, \vartheta)$, see (\ref{A100}), and hence is a
continuation of $k_t$ to $[T_1, T_2]$. Now we repeat this procedure
due times and obtain
\begin{equation*}
  k^{(n)}_t = Q_{\vartheta_n \vartheta_{T_{n-1}}} (t) k_{T_{n-1}},
\end{equation*}
where $\vartheta_n = \vartheta_{T_{n-1}} +1$ and
\begin{gather}
  \label{A58}
T_n = \kappa \tau(\vartheta_{T_{n-1}})  \\[.2cm]
 \vartheta_{T_n} = \vartheta_{T_{n-1}} + \log \left( 1 +
T_{n-1}\|b\|
 e^{-\vartheta_{T_n-1}}\right), \quad \vartheta_{T_0} := \vartheta_0. \nonumber
\end{gather}
The continuation to all $t>0$ will be obtained if we show that
$\sum_{n\geq 1} T_n = +\infty$. Assume that this is not the case.
From the second line in (\ref{A58}) we get $T_{n-1} =
(e^{\vartheta_{T_n}} - e^{\vartheta_{T_{n-1}}})/\|b\|$. Hence
\[
\sum_{n=1}^N T_n = \left(e^{\vartheta_{T_N}} -
e^{\vartheta_0}\right)/ \|b\|.
\]
Thus, the mentioned series converges if the sequence
$\{\vartheta_{T_n}\}_{n\in \mathds{N}}$ is bounded, say by
$\bar{\vartheta}$. However, in this case one cannot get $T_n \to 0$
as $n\to +\infty$, for it contradicts the first line in (\ref{A58})
since
$$\tau(\vartheta)\geq \left[ e\langle a \rangle e^{\bar{\vartheta}}
+ \|b\|e^{-\vartheta_0} \right]^{-1}, $$ see (\ref{A55}). Clearly,
the upper bound in (\ref{T10}) hold on each $[T_{n-1}, T_{n}]$. This
completes the proof.

\section{The Global Boundedness}
Here we prove Theorem \ref{2tm}. Let $\varDelta$ be a closed cubic
cell such that
\begin{equation}
 \label{J22}
\inf_{x\in \varDelta}a (x) =: a_\varDelta >0,
\end{equation}
which is possible since $a$ is continuous.  For $\eta$ contained in
a translate of $\varDelta$, $|\eta|\geq 2$, and $x\in \eta$, we then
have
\begin{equation}
 \label{J23}
\sum_{y \in \eta \setminus x} a (x-y) \geq a_\varDelta
\left(|\eta|-1\right)\geq a_\varDelta .
\end{equation}
For a translate of $\varDelta$, we consider the observables
$N_\varDelta^n: \Gamma \to \mathds{N}_0$ defined as follows:
 $N^n_\varDelta(\gamma)
=|\gamma_\varDelta|^n$,  $n\in \mathds{N}$. Then
\begin{eqnarray}
  \label{J27}
N_\varDelta(\gamma) & = & \sum_{x\in \gamma} I_\varDelta(x), \\[.2cm]
N^n_\varDelta(\gamma) & = & \sum_{l=1}^n l! S(n,l) \sum_{\{x_1 ,
\dots , x_l\} \subset \gamma} I_\varDelta(x_1) \cdots
I_\varDelta(x_l), \quad n\geq 2, \nonumber
\end{eqnarray}
where $I_\varDelta$ is the indicator function of $\varDelta$ and
$S(n,l)$ is a Stirling numbers of the second kind, equal to the
number of distinct ways of dividing $n$ labeled items into $l$
unlabeled groups. It has the following representation, cf
\cite{Riordan},
\begin{equation}
\label{Tou1} S(n,l) = \frac{1}{l!}\sum_{s=0}^l \left(-1\right)^{l-s}
\binom{l}{s} s^n.
\end{equation}
Then, for $\mu\in \mathcal{P}_{\rm exp}(\Gamma)$, by
(\ref{1fa}) we have that
\begin{equation}
  \label{J28}
  \mu(N_\varDelta^n) = \sum_{l=1}^n S(n,l) \int_\varDelta \cdots \int_\varDelta k^{(l)}_\mu(x_1 ,
  \dots , x_l) d x_1 \cdots d x_l.
\end{equation}
For $l\in \mathds{N}$, we set
\begin{eqnarray}
  \label{J27a}
F_\varDelta^{(l)} (\gamma ) & = & \sum_{\{x_1 , \dots , x_l\}
\subset
\gamma} I_\varDelta(x_1) \cdots I_\varDelta(x_l)\\[.2cm] \nonumber
& = & \frac{1}{l!} N_\varDelta(\gamma) \left( N_\varDelta(\gamma) -
1\right)\cdots \left( N_\varDelta(\gamma) - l+1\right).
\end{eqnarray}
And also $F_\varDelta^{(0)} (\gamma )\equiv 1$. Then we can rewrite
(\ref{J27}) as follows
\begin{eqnarray}
  \label{J29}
N^n_\varDelta(\gamma)  =  \sum_{l=1}^n l! S(n,l) F_\varDelta^{(l)}
(\gamma ).
\end{eqnarray}
An easy calculation yields
\begin{equation}
  \label{J30}
F_\varDelta^{(l)} (\gamma \cup x) - F_\varDelta^{(l)} (\gamma ) =
I_\varDelta (x) F_\varDelta^{(l-1)} (\gamma ),
\end{equation}
For $\mu_t$ as in Theorem \ref{1tm}, we set
\begin{equation}
  \label{J31}
 q_\varDelta^{(0)} (t) \equiv 1, \qquad q_\varDelta^{(l)} (t) = \mu_t (F_\varDelta^{(l)}), \qquad l\in \mathds{N}.
\end{equation}
By (\ref{J28}), (\ref{J27a}) and (\ref{J29}) it follows that
\begin{equation}
  \label{J31c}
q_\varDelta^{(l)} (0) = \frac{1}{l!} \int_\varDelta \cdots
\int_\varDelta k^{(l)}_{\mu_0} (x_1 , \dots , x_l) d x_1 \cdots d
x_l.
\end{equation}
Since $\mu_0$ is in $\mathcal{P}_{\rm exp} (\Gamma)$, one finds
$\vartheta \in \mathds{R}$ such that $k^{(l)}_{\mu_0} (x_ 1, \dots,
x_l) \leq e^{\vartheta}$, cf (\ref{I4}). By (\ref{J31c}) this yields
\begin{equation}
  \label{J31d}
q_\varDelta^{(l)} (0) \leq \left[ {\rm V}(\varDelta)
e^\vartheta\right]^l/l!, \qquad l\in \mathds{N}.
\end{equation}
Recall that $a_\varDelta$ is defined in (\ref{J22}), see also
(\ref{J23}).  Set
\begin{gather}
 \label{J31e}
 b_\varDelta = \int_{\varDelta} b(x) dx, \qquad
 \kappa_\varDelta= \max\left\{ {\rm V}(\varDelta) e^\vartheta; {b_\varDelta}/{
a_\varDelta}\right\},
\end{gather}
where $\vartheta$ is as in (\ref{J31d}).

The proof of the lemma below is based on the following elementary
arguments. Let $u:[0,+\infty) \to \mathbb{R}$ be continuously
differentiable with the derivative satisfying
\begin{equation}
  \label{J24}
  u'(t) \leq b - a u(t), \qquad a, b >0.
\end{equation}
Then
\[
u(t) \leq u(0) e^{-at} + \frac{b}{a} \left( 1 - e^{-at}\right),
\qquad t\geq 0,
\]
which, in particular, means that
\begin{eqnarray}
  \label{J25}
u(t) \leq \max\{ u(0); b/a \},
\end{eqnarray}
and also: for each $\varepsilon >0$, there exists
$\tau_\varepsilon\geq 0$ such that
\begin{equation}
  \label{J26}
  \forall t\geq \tau_\varepsilon \qquad u(t) \leq \varepsilon + b/a.
\end{equation}
\begin{lemma}
  \label{J3lm}
Let $\varDelta$ be as in (\ref{J22}) and $\mu_t$, $t\geq 0$ be as in
Theorem \ref{1tm}, and hence $q_\varDelta^{(l)} (0)$ satisfies
(\ref{J31d}) with some $\vartheta$. Let $\kappa_\varDelta$ be as in
(\ref{J31e}) for these parameters. Then
\begin{equation}
  \label{J31a}
\forall t \geq 0 \qquad   q_\varDelta^{(l)} (t) \leq
\kappa^l_\varDelta/l!, \qquad l\in \mathds{N}.
\end{equation}
\end{lemma}
\begin{proof}
By (\ref{R2}) we have that
\begin{equation*}
 \frac{d}{dt}  q_\varDelta^{(l)} (t) = \mu_t (L F_\varDelta^{(l)}),
\end{equation*}
which by means of (\ref{J30}) can be written
\begin{eqnarray}
  \label{J33}
\frac{d}{dt}  q_\varDelta^{(l)} (t) & = & b_\varDelta
q_\varDelta^{(l-1)} (t) - \int_{\Gamma} \left( \sum_{x\in
\gamma_\varDelta}\left( \sum_{y\in \gamma\setminus x} a(x-y) \right)
F^{(l-1)}_\varDelta
(\gamma\setminus x) \right) \mu_t (d \gamma) \qquad \\[.2cm] \nonumber
& \leq & b_\varDelta  q_\varDelta^{(l-1)} (t) - \int_{\Gamma} \left(
\sum_{x\in \gamma_\varDelta} \left( \sum_{y\in \gamma_\varDelta
\setminus x} a(x-y) \right)F^{(l-1)}_\varDelta
(\gamma\setminus x) \right) \mu_t (d \gamma) \\[.2cm] \nonumber
& \leq & b_\varDelta  q_\varDelta^{(l-1)} (t) -  a_\varDelta
\int_{\Gamma} \left( \sum_{x\in \gamma_\varDelta}
F^{(l-1)}_\varDelta (\gamma\setminus x) \right) \mu_t (d \gamma).
\end{eqnarray}
By (\ref{J29}) it follows that
\[
\sum_{x\in \gamma_\varDelta}  F^{(l-1)}_\varDelta (\gamma\setminus
x)  = l F^{(l)}_\varDelta (\gamma).
\]
We apply this in (\ref{J33}) and obtain, cf (\ref{J24}) and
(\ref{J31e}),
\begin{equation}
  \label{J35}
\frac{d}{dt}  q_\varDelta^{(l)} (t) \leq b_\varDelta
q_\varDelta^{(l-1)} (t) -
 l  a_\varDelta
 q_\varDelta^{(l)} (t), \qquad l\in \mathds{N}.
\end{equation}
For $l=1$, by (\ref{J31}) and (\ref{J25}) we get from the latter that (\ref{J31a}) holds.
Now we assume that (\ref{J31a}) holds for a given $l-1$. It yields in (\ref{J35})
\[
 \frac{d}{dt}  q_\varDelta^{(l)} (t) \leq  \frac{b_\varDelta \kappa_\varDelta^{l-1}}{(l-1)!} - l  a_\varDelta
 q_\varDelta^{(l)} (t),
\]
from which by (\ref{J25}) we obtain that (\ref{J31a}) holds also for
$l$.
\end{proof}
\noindent {\it Proof of Theorem \ref{2tm}.}  By means of the evident
monotonicity
\[
\mu(N_\Lambda^n) \leq \mu(N_\Lambda^{n+1}), \qquad \mu(N_\Lambda^n)
\leq \mu(N_{\Lambda'}^n), \ \quad {\rm for} \ \ \Lambda \subset
\Lambda',
\]
we conclude that it is enough to prove the statement for: (a)
$n=2^s$; (b) $\Lambda$ being a finite sum of the translates of the
cubic cell $\varDelta$ as in Lemma \ref{J3lm}. Let $m$ be such that
\begin{equation*}
 \Lambda =
\bigcup_{l=1}^m \varDelta_{l},
\end{equation*}
where the cells $\varDelta_l$ are such that the intersection of each
two distinct ones is of zero Lebesgue measure.
 By the estimate
\[
\left(\sum_{l=1}^n a_l \right)^2 \leq n \sum_{l=1}^n a_l^2,
\]
we prove that
\begin{equation*}
  N_\Lambda^{2^s}(\gamma) \leq m^{2^s -1} \sum_{l=1}^m N_{\varDelta_l}^{2^s}
  (\gamma), \qquad s\in \mathds{N}_0.
\end{equation*}
Then by Lemma \ref{J3lm} and (\ref{J29}) we obtain
\begin{equation*}
\mu_t (N_\Lambda^{2^s}) \leq m^{2^s} T_{2^s}
(\bar{\kappa}_\varDelta) = \left[{\rm V}(\Lambda) \right]^{2^s}
\left(T_{2^s} (\bar{\kappa}_\varDelta)/\left[{\rm V}(\varDelta)
\right]^{2^s}\right),
\end{equation*}
where $T_n$ is the Touchard polynomial
\begin{equation*}
T_n(\varkappa) := \sum_{l=1}^n S(n, l) \varkappa^l,
\end{equation*}
with $S(n, l)$ given in (\ref{Tou1}),  see \cite[eq.
(3.4)]{Riordan},
 and
\begin{equation*}
\bar{\kappa}_\varDelta := {\rm V}(\varDelta) \max\left\{
e^\vartheta; {\|b\|}/{a_\varDelta}\right\},
\end{equation*}
cf (\ref{J31e}). This proves (\ref{T200}).

If the initial state $\mu_0$ is such that each $k^{(l)}_{\mu_0}\in
C_{\rm b} ((\mathds{R}^{d})^l)$ -- the set of bounded continuous
functions, then so is $k^{(l)}_{t}$ for all $t>0$. This can be
proved by repeating the corresponding proof in \cite{FKK}. As in
(\ref{T6}) we have, see also (\ref{J28}),
\begin{equation*}
\mu_t (N_\varDelta) = \int_{\varDelta}k^{(1)}_t (x) dx.
\end{equation*}
By taking a
sequence of $\varDelta$ shrinking up to a given $x$ and applying
(\ref{J31a}) we obtain
\begin{equation}
  \label{J39}
 k_t^{(1)} (x) \leq \max\left\{ k_{\mu_0}^{(1)}(x); b(x)/a(0)
 \right\} \leq \max\{ \|k_{\mu_0}^{(0)}\|_{L^\infty(\mathds{R}^d)}; \|b\|/a(0) \},
\end{equation}
which proves the bound in (\ref{J1d}) for $k^{(1)}_t$. Note that the
smaller bound in (\ref{J39}) coincides with the corresponding bound
in the exactly soluble Surgailis model in which the mortality rate
$m(x)$ is substituted by $a (0)$. That is, the competition here
amounts to the appearance of an effective mortality $a(0)$. Another
important observation regarding the competition in this model is
based on (\ref{J26}). Let $\Lambda$ be compact and $k_t^{(1)}$
satisfy (\ref{J39}). Then for an arbitrary $\varepsilon>0$, one
finds $\tau(\varepsilon, \Lambda)$, dependent also on $\mu_0$, such
that, for all $x\in \Lambda$ and $t\geq \tau(\varepsilon, \Lambda)$,
the following holds
\begin{equation*}
k_t^{(1)} (x) \leq \frac{b(x)}{ a(0)} +\varepsilon.
\end{equation*}
That is, after some time
the density at each point of $\Lambda$ approaches a certain level, independent of the
initial distribution of the entities in $\Lambda$.

Let us now prove the validity of the second estimate in (\ref{J1d}). The bound for $k^{(2)}_t (x,x)$ can be obtained
from (\ref{J31a}) in the way similar to that used in getting (\ref{J39}). To bound $k^{(2)}_t (x,y)$ with $x\neq y$, let us take two
disjoint cells $\varDelta_x$ and $\varDelta_y$. Both are  of side $h>0$ and such that: (a) $x\in \varDelta_x$ and $y\in \varDelta_y$; (b)
$\varDelta_x \to \{x\}$ and $\varDelta_y \to \{y\}$ as $h \to 0$. Then we set
\begin{equation*}
F_h (\gamma) = \left[ \sum_{z\in \gamma}\left(I_{\varDelta_x} (z) - I_{\varDelta_y}(z) \right) \right]^2.
\end{equation*}
For the state $\mu_t$, we have
\begin{gather*}
 0 \leq \mu_t (F_h) = q^{(2)}_{\varDelta_x}(t) +
 q^{(2)}_{\varDelta_y}(t) - 2 \int_{\varDelta_x}\int_{\varDelta_y}
 k^{(2)}_t (z_1 , z_2) d z_1 d z_2.
\end{gather*}
By (\ref{J31e}) and (\ref{J31a}) this yields
\begin{gather*}
  \int_{\varDelta_x}\int_{\varDelta_y}
 k^{(2)}_t (z_1 , z_2) d z_1 d z_2 \leq
 \frac{1}{2}\max\{\kappa_{\varDelta_x}^2; \kappa_{\varDelta_y}^2\} \leq
 \frac{h^{2d}}{2} \max\{ e^{2\vartheta}; (\|b\|/ a_h)^2\},
\end{gather*}
where $a_h = \min\{ a_{\varDelta_x}; a_{\varDelta_y}\}$. Passing
here to the limit $h\to 0$ and taking into account the continuity of
$k_t^{(2)}$ and $a$ we obtain the second inequality in (\ref{J1d}).
This completes the proof. \hskip12.5cm $\square$

\section*{Acknowledgment}
The present research was supported by the European Commission under
the project STREVCOMS PIRSES-2013-612669 and by the SFB 701
``Spektrale Strukturen and Topologische Methoden in der Mathematik".

\end{document}